\documentclass[11pt]{amsart} 
\usepackage{latexsym}
\usepackage{amsfonts}
\usepackage{amsmath,amsthm,amssymb}

\newcommand{\cal}{\mathcal}

\def\epsilon{\varepsilon}
\def\phi{\varphi}

\def\hat{\widehat}

\newcommand{\Out}{\mbox{\rm Out}}
\newcommand{\Aut}{\mbox{\rm Aut}}

\newcommand{\FN}{F_N}   

\newcommand{\CVN}{\mbox{\rm CV}_N}
\newcommand{\CVNbar}{\overline{\mbox{\rm CV}}_N}

\newcommand{\Z}{\mathbb Z}

\def\strutdepth{\dp\strutbox}
\def \ss{\strut\vadjust{\kern-\strutdepth \sss}}
\def \sss{\vtop to \strutdepth{
\baselineskip\strutdepth\vss\llap{$\diamondsuit\;\;$}\null}}

\def\strutdepth{\dp\strutbox}
\def \sst{\strut\vadjust{\kern-\strutdepth \ssss}}
\def \ssss{\vtop to \strutdepth{
\baselineskip\strutdepth\vss\llap{$\spadesuit\;\;$}\null}}

\def\strutdepth{\dp\strutbox}
\def \ssh{\strut\vadjust{\kern-\strutdepth \sssh}}
\def \sssh{\vtop to \strutdepth{
\baselineskip\strutdepth\vss\llap{$\heartsuit\;\;$}\null}}

\def\qed{\hfill\rlap{$\sqcup$}$\sqcap$\par}
\def\bar{\overline}
\def\tilde{\widetilde}

\def\strutdepth{\dp\strutbox}
\def \ss{\strut\vadjust{\kern-\strutdepth \sss}}
\def \sss{\vtop to \strutdepth{
\baselineskip\strutdepth\vss\llap{$\diamondsuit\;\;$}\null}}

\def\strutdepth{\dp\strutbox}
\def \sst{\strut\vadjust{\kern-\strutdepth \ssss}}
\def \ssss{\vtop to \strutdepth{
\baselineskip\strutdepth\vss\llap{$\spadesuit\;\;$}\null}}

\def\qed{\hfill\rlap{$\sqcup$}$\sqcap$\par}

\newtheorem{thm}{Theorem}[section]
\newtheorem{cor}[thm]{Corollary}
\newtheorem{lem}[thm]{Lemma}
\newtheorem{prop}[thm]{Proposition}

\theoremstyle{definition}
\newtheorem{defn}[thm]{Definition}
\newtheorem{example}[thm]{Example}

\newtheorem{rem}[thm]{Remark}

\newtheorem{defn-rem}[thm]{Definition-Remark}

\theoremstyle{remark}

\numberwithin{equation}{section}

\begin{document} 
 
%
 

\author[M.~Lustig]{Martin Lustig}
\author[K.~Ye]{Kaidi Ye}
\address{\tt 
Aix Marseille Universit\'e, CNRS, Centrale Marseille, I2M UMR 7373,
13453  Marseille, France
}
\email{\tt Martin.Lustig@univ-amu.fr}
\email{\tt deloresye@gmail.com}

 
\title[Quadratic growth automorphisms of $\FN$]
{Normal form and  parabolic dynamics 
for quadratically growing automorphisms of free groups}


\begin{abstract} 
We present a normal form for outer automorphisms $\phi$ of a non-abelian free group $\FN$ 
which grow quadratically (measured through the maximal growth of conjugacy classes in $\FN$ under iteration of $\phi$). In analogy to the known normal form for 
linearly
growing automorphisms as {\em efficient Dehn twist}, our normal form for $\phi$ is given in terms of a 2-level 
Dehn twist on a 
graph-of-groups $\cal G$ with $\pi_1 \cal G \cong  \FN$, where a conjugacy class of $\FN$ grows at most linearly if and only if it is contained in a vertex group of $\cal G$. 

Our proof is based on earlier work of the second author \cite{Ye01, Ye02, Ye03} and on a new cancellation result, which also allows us to show that the dynamics of the induced $\phi$-action on Outer space $\CVN$ consists entirely of parabolic orbits, with limit points all assembled in the simplex $\Delta_\cal G \subset \partial \CVN$ determined by  $\cal G$.
\end{abstract}

 

\maketitle 

\section{Introduction}

Outer automorphisms of a free group $\FN$ of finite rank $N \geq 2$ have 
received 
a lot of attention in the past 30 years, since the groundbreaking papers of Culler-Vogtmann \cite{CV} and Bestvina-Handel \cite{BH92}. Much progress has been obtained, in particular from the attempt to mimic important known features from mapping classes and from the action of the mapping class group 
$\text{Mod}_g$ on Teichm\"uller space $\cal T_g$. However, in many aspects automorphisms of free groups can be intrinsically more complicated than mapping classes, and the group $\Out(\FN)$ is less tractable and 
more mysterious 
than $\text{Mod}_g$. In addition, the natural analogue of $\cal T_g$ with its canonical $\text{Mod}_g$-action, namely Outer space $\CVN$ equipped with a canonical $\Out(\FN)$-action, is not a manifold and hence immune towards all attempts 
to mimic 
directly 
the well developed analytic theory for $\cal T_g$.

\smallskip

One of the most obvious differences to $\text{Mod}_g$ is that $\Out(F_N)$ contains elements which grow polynomially  
of degree $d \geq 2$.
Here the growth function $\text{Gr}_\phi(t)$ of an outer automorphism $\phi \in \Out(\FN)$ is given by considering, for any element $w \in \FN$, the function
$\|\phi^t([w])\|_\cal A$, 
where $\cal A$ is any basis of $\FN$, and $\|[u]\|_\cal A$ denotes the length of a cyclically reduced word 
in
$\cal A^{\pm1}$ which represents the conjugacy class $[u] \subset \FN$.
The choice of $\cal A$ is immaterial if one is only interested in the type of the  function $\text{Gr}_\phi(t)$, which is taken as maximum 
over the above functions, for any $w \in \FN$. 
It is 
a well known consequence of \cite{BH92} 
that $\text{Gr}_\phi(t)$ is either an exponential function with growth rate given by a Perron-Frobenius transition matrix derived from $\phi$, or else $\text{Gr}_\phi(t)$ must be a polynomial 
of some degree
$d \geq 0$.

Exponentially growing automorphisms $\phi \in \Out(\FN)$ have received on the whole more attention than polynomially growing ones, as they occur more frequently. In addition, polynomially growing such $\phi$ are technically often harder to deal with, since there is less rigidity in the intrinsic structure of such automorphisms. Nevertheless, the only type of automorphism of $\FN$ for which a normal from was available so far are linearly growing automorphisms: 

In two joint papers \cite{CL95, CL99} with Marshall Cohen the first author derived this normal form by exhibiting, for any linearly growing $\phi \in \Out(\FN)$, a certain 
type of graph-of-groups $\cal G$ with a marking isomorphism $\pi_1 \cal G \cong \FN$, together with a graph-of-groups automorphism $D: \cal G \to \cal G$, called an {\em efficient Dehn twist} 
(see section 
\ref{2.4}).  
The map $D$ 
induces for some integer $m \geq 1$ the outer automorphism $\hat D = \phi^m$ on $\pi_1 \cal G$, 
and as such is unique up to graph-of-groups isomorphisms.
It was later shown in joint work with S. Krstic and K. Vogtmann \cite{KLV01} that $\phi$ itself is also induced by a graph-of-groups automorphism $R: \cal G \to \cal G$ with $R^m = D$.

\smallskip

In this paper the authors follow very much the same strategy for quadratically growing automorphisms: We define {\em efficient 2-level Dehn twists} as 
graph-of-groups
automorphisms $H: \cal G \to \cal G$, but while for the above $D$ the induced {\em local automorphisms} on the vertex groups are the identity, the vertex group automorphisms of $H$ are themselves given through efficient Dehn twists. 
The conditions imposed by our Definition \ref{2-level-efficient} ensure that $\hat H$ has always quadratic growth. 

It 
follows from the results of 
\cite{Ye02, Ye03} that every quadratically growing $\phi \in \Out(\FN)$
has a positive power which
can be represented by some 2-level Dehn twist, and through suitable modifications 
(see section \ref{sec-3}) the latter can be made efficient. Much harder is the question about uniqueness, and again we follow closely the method employed by \cite{CL95, CL99}: There it has been shown that every Dehn twist automorphism induces on $\CVN$ an action with parabolic orbits, where the limit point of any orbit is contained in 
a 
simplex $\Delta_\cal G$ in the boundary $\partial \CVN$, which in turn is given by varying the edge lengths of the graph $\cal G$ on which the efficient Dehn twist $D$ had been defined. 

Exactly the same is shown in our Theorem \ref{parabolic-dynamics}, quoted here for simplicity without the detailed information about the limit points. However, it should be noted that the proof in the quadratic growth case is substantially harder than in the linear 
case: We need to employ the concept of {\em $H$-conjugacy} (see section \ref{2.3}) to the {\em correction terms} of the edges of $\cal G$ (see Definition \ref{gog-iso}), and 
the proof given here uses crucially a growth result about the $H$-conjugacy classes of those correction terms, proved previously by the second author in \cite{Ye02}.

\begin{thm}
\label{intro-parabolic-dynamics}
Let 
$[\Gamma]$
be any point in Outer space $\CVN$, given by a 
marked metric graph $\Gamma$.
Then for any automorphism $\phi \in \Out(\FN)$, 
represented by an efficient 2-level Dehn twist 
$H: \cal G \to \cal G$, 
the $\phi$-orbit of 
$[\Gamma]$ 
is parabolic, with 
limit point 
contained 
in the 
interior of the 
simplex $\Delta_\cal G \subset \partial\CVN$.
\end{thm}

A first convergence result for the action of quadratically (or higher-degree polynomially) growing automorphism $\phi \in \Out(\FN)$ on $\CVNbar = \CVN \cup \partial \CVN$ has been obtained by M. Bestvina, M. Feighn and M. Handel's in Theorem 1.4 of \cite{BFH05}.

From 
Theorem \ref{intro-parabolic-dynamics} 
we derive the desired uniqueness, thus justifying our terminology ``normal form'' (see Theorem \ref{normal-form}).
Related results have been obtained by M. Rodenhausen in \cite{R} (see Remarks \ref{Rod1} and \ref{Rod2} below).

\begin{thm}
\label{intro-normal-form}
Two efficient 2-level Dehn twists $H:\cal G \to \cal G$ and $H':\cal G' \to \cal G'$ represent outer automorphisms $\hat H$ and $\hat H'$ of a free group $\FN$ which are conjugate in $\Out(\FN)$ if and only if there exists a graph-of-groups isomorphism $F: \cal G \to \cal G'$ which satisfies:
$$\hat H = \hat F^{-1} \hat H' \hat F$$
\end{thm}

It  turns out that the extension of the above normal form to {roots of 2-level Dehn twists} 
is easier than in the linear case, since contrary to that case, for 2-level Dehn twists $H: \cal G \to \cal G$ the edge groups of $\cal G$ are trivial. We obtain (see Theorem \ref{roots-normal-form}):

\begin{cor}
\label{intro-roots-normal-form}
(1)
Every automorphism $\phi \in \Out(\FN)$ with exponent 
$m \geq 1$, 
such that $\phi^m$ is represented by an efficient 2-level Dehn twist $H: \cal G \to \cal G$, can be represented by a graph-of-groups automorphism $R: \cal G \to \cal G$.

\smallskip
\noindent
(2)
Two graph-of-groups automorphism $R: \cal G \to \cal G$ and $R': \cal G' \to \cal R'$ as in part (1) represent outer automorphisms $\hat R$ and $\hat R'$ of a free group $\FN$ which are conjugate in $\Out(\FN)$ if and only if there exists a graph-of-groups isomorphism $F: \cal G \to \cal G'$ which satisfies:
$$\hat R = \hat F^{-1} \hat R' \hat F$$
\end{cor}

There are a number of obvious algorithmic questions issuing 
from 
the above results; they will be answered in the forthcoming joint work \cite{LY2}.


\medskip

\noindent
{\em Acknowledgements:} Both authors would like to thank Arnaud Hilion for several helpful discussions.


\section{Preliminaries}
\label{prelims}
\subsection{Graphs-of-groups and their isomorphisms}
\label{2.1}

${}^{}$


\smallskip

In this subsection we set up the basic notation while recalling some fundamental facts about graph-of-groups and their isomorphisms. For more details on graphs-of-groups we refer the reader to \cite{Bass93, CL99, Levitt,  Serre80}.

\smallskip

Unless otherwise stated, a {graph} $\Gamma$ in this paper is finite, non-empty and connected. 
We denote the {vertex set} of $\Gamma$ by $V(\Gamma)$ and the set of oriented { edges by $E(\Gamma)$. For any edge $e$ in $E(\Gamma)$ we denote by $\tau(e)$ its {terminal vertex}, and by $\bar{e}$ the edge with reversed orientation. Hence the initial vertex of $e$ is given by $\tau(\bar e)$.

Our graph $\Gamma$ is non-oriented, but one can always choose an {\it orientation} of $\Gamma$, given as subset $E^{+}(\Gamma) \subset E(\Gamma)$ such that $E^{+}(\Gamma) \cup \bar{E}^{+}(\Gamma)=E(\Gamma)$ and $E^{+}(\Gamma) \cap \bar{E}^{+}(\Gamma)=\emptyset$, where $\bar{E}^{+}(\Gamma)=\{\bar{e} \mid e \in E^{+}(\Gamma)\}$.

\begin{defn}
A {\it graph-of-groups} $\cal{G}$ is given by the following data: A graph $\Gamma$, a {\it vertex group} $G_v$ for each $v \in V(\Gamma)$, an {\it edge group} $G_e$ for each $e \in E(\Gamma)$, with $G_e = G_{\bar{e}}$, and an injective {\it edge homomorphism} $f_e: G_e \rightarrow G_{\tau(e)}$ for every edge $e$ of $\Gamma$.
\end{defn}

Given a graph-of-groups $\cal{G}$, we usually denote its underlying graph by $\Gamma(\cal{G})$, while the vertex set and edge set of $\Gamma(\cal{G})$ are denoted by $V(\cal{G})$ and $E(\cal{G})$ respectively.
To each edge $e \in E(\cal{G})$ we abstractly associate a {\it stable letter}  $t_e$.

If the underlying graph $\Gamma(\cal G)$ consists of a single vertex only, then the graph-of-groups $\cal G$ is sometimes called {\em trivial}.

\begin{defn}
(1)
For any graph-of-groups $\cal{G}$ the {\it word group} $W(\cal{G})$ is defined to be the free product of the vertex groups and the free group generated by all stable letters:
$$W(\cal{G})= \big (\underset{v \in V(\Gamma)}{\operatorname{\ast}}G_{v}\big )*F(\{t_{e} \mid e\in E(\Gamma)\})$$

\smallskip
\noindent
(2)
The {\it path group} $\Pi(\cal{G})$ is defined to be the quotient of $W(\cal{G})$ modulo the relations $t_{\bar{e}}=t_{e}^{-1}$ and $f_{\bar{e}}(g)=t_{e}f_{e}(g)t_{e}^{-1}$ for all $e \in E(\cal{G})$ and $g \in G_{e}$.
\end{defn}

\begin{rem}
\label{gog-normal-form}
(1)
Since $\Pi(\cal G)$ is vastly more important than $W(\cal G)$, any {\em word} $W  = r_{0}t_{1}r_{1}...r_{q-1}t_{q}r_{q}$, with $t_{i} = t_{e_i}$ for some $e_i \in E(\cal{G})$ and $r_{i} \in\underset{v \in V(\Gamma)}{\operatorname{\ast}}G_{v}$, though formally an element in $W(\cal G)$, will always be understood as element in $\Pi(\cal G)$ (unless explicitly stated otherwise). In particular, if $W'$ is a second such word, then $W = W'$ means that they are equal in $\Pi(\cal G)$.

\smallskip
\noindent
(2) This is justified by the following ``normal form'' in $\Pi(\cal G)$:

Let $W  = r_{0}t_{1}r_{1}...r_{q-1}t_{q}r_{q}$ and 
$W'  = r'_{0}t'_{1}r'_{1}...r'_{q'-1}t'_{q'}r'_{q'}$
be two words in $W(\cal G)$. Then $W$ and $W'$ define the same element in $\Pi(\cal G)$ if and only if $q = q'$, and if for any 
$k = 1, \ldots, q$
one has $t'_k = t_k$  and there exist elements $g_k \in G_{e_k}$ such that the equalities $r'_{k} = f_{e_{k}}(g_{k}) r_{k} f_{\bar e_{k+1}}(g_{k+1})$ for $k \neq q$ as well as $r'_{0} =  r_{0} f_{\bar e_{1}}(g_{1})$ and $r'_{q} = f_{e_{q}}(g_{q}) r_{q}$ hold.

\smallskip
\noindent
(3) 
As a consequence, the
{\it path length} 
(or {\it $\cal{G}$-length}) 
of any word $W  = r_{0}t_{1}r_{1}...r_{q-1}t_{q}r_{q}$, given by 
$$|W|_{\cal{G}} = q \, ,$$
is a well defined notion in $\Pi(\cal G)$.
If the context is unambiguous, we sometimes write $| W | $ for $ | W |_\cal{G}$.
\end{rem}

\begin{defn}
\label{reduced-word}
A word $W  = r_{0}t_{1}r_{1}...r_{q-1}t_{q}r_{q}$ in $W(\cal G)$
%
 is said to be 
\begin{enumerate}
\item {\it connected} if the sequence $e_1 e_2 ...e_q$ 
(for $t_k = t_{e_k}$)
forms a connected path 
$\gamma$, and if $r_{0} \in G_{\tau(\bar{e}_{1})}$ and $r_{i} \in G_{\tau(e_{i})}$ for $1 \leq i \leq q$.
In this case we write $\tau(W) := \tau(\gamma) := \tau(e_q)$, and call it the {\em terminal vertex} of the word $W$ or of the path $\gamma$.
Similarly, their {\em initial vertex} is given by $\tau(\bar e_1)$.

\item {\it closed connected} if it is connected and $\tau(\bar e_1)=\tau(e_q)$. 
In order to specify the initial vertex 
we sometimes call $W$ a {\it closed connected word issued at $\tau(\bar e_1)$}.
\item {\it reduced} if $q=0$ 
and $r_0 \neq 1$, 
%
or if, in case $q>0$, whenever $t_{i} = t_{i+1}^{-1}$ for some $1 \leq i \leq q-1$ we have $r_{i}\not\in f_{e_{i}}(G_{e_{i}})$.
\item {\it cyclically reduced}: if it is reduced and if, in case $q>0$ and $t_{1} = t_{q}^{-1}$, one has $r_{q}r_{0}\not\in f_{e_{q}}(G_{e_{q}})$.
\end{enumerate}
\end{defn}

It follows from Remark \ref{gog-normal-form} (2) that the terminology introduced in the last definition applies as well to the element in $\Pi(\cal G)$ defined by the word $W$ in $W(\cal G)$.

\begin{defn}
For any graph-of-groups $\cal{G}$ and any vertex $v \in V(\cal{G})$, we denote by $\pi_{1}(\cal{G}, v)$ the {\it fundamental group based at $v$}, which consists of all elements in $\Pi(\cal{G})$ that are represented by closed connected words issued at $v$.
\end{defn}
 
For distinct vertices $v_1, v_2 \in V(\cal{G})$, the fundamental groups $\pi_{1}(\cal{G}, v_1)$ and $\pi_1(\cal{G}, v_2)$ are conjugate in $\Pi(\cal{G})$. Sometimes, we write $\pi_1(\cal{G})$ when the base point does not make a difference.

\begin{defn}
\label{gog-iso}
Given two graphs-of-groups $\cal{G}_1$ and $\cal{G}_2$,
a {\it graph-of-groups isomorphism} $H: \cal{G}_1 \rightarrow \cal{G}_2$ consists of 
\begin{enumerate}
\item a graph isomorphism $H_{\Gamma}: \Gamma(\cal G_1) \rightarrow \Gamma(\cal G_2)$, 
\item a group ismorphism $H_{v}: G_{v} \rightarrow G_{H_{\Gamma}(v)}$ for each vertex $v \in V(\cal G_{1})$,
\item a group isomorphism $H_{e}=H_{\bar{e}}: G_{e} \rightarrow G_{H_{\Gamma}(e)}$ for each edge $e \in E(\cal G_{1})$, and
\item for every $e \in E(\cal G_{1})$ an element $\delta(e) = 
\delta_H(e) \in G_{\tau(H_{\Gamma}(e))}$, called the {\it correction term} for $e$, which satisfies:
\begin{equation}
\label{gog-isom-eq}
H_{\tau(e)}f_{e}=ad_{\delta(e)}f_{H_{\Gamma}(e)}H_{e}
\end{equation}
\end{enumerate}
Here and below we denote by $ad_{g}$ the inner automorphism  $x  \mapsto gxg^{-1}$.
\end{defn}

The isomorphism $H: \cal{G}_{1} \rightarrow \cal{G}_{2}$ induces an isomorphism $H_{*}: \Pi(\cal{G}_{1}) \rightarrow \Pi(\cal{G}_{2})$ defined on the generators by
\begin{enumerate}
\item[(a)] $H_{*}(g) = H_{v}(g)$ for any $g \in G_{v}$ and any $v \in V(\cal G_{1})$, and 
\item[(b)] $H_{*}(t_{e}) = \delta(\bar{e}) t_{H_{\Gamma}(e)} \delta(e)^{-1}$ for any $e \in E(\cal G_{1})$.
\end{enumerate}

For every $v \in V(\cal{G})$, the isomorphism $H_*$ induces an isomorphism $H_{*v}: \pi_1(\cal{G}_1, v) \rightarrow \pi_1(\cal{G}_2, H_{\Gamma}(v))$. 
%
We denote by 
$\hat{H}$
the {outer isomorphism} induced by $H_*$.
Here we use the terminology introduced in
\cite{CL99}:
for arbitrary 
groups $G_1$ and $G_2$
an {\it outer isomorphism} $\hat{f}$ 
is an equivalence class of isomorphism, where two isomorphism $f: G_1 \rightarrow G_2$ and $f': G_1 \rightarrow G_2$ are equivalent if 
there is 
an element $g \in G_2$ such that $f'= ad_{g} \circ f$ (for $ad_g$ as defined above).

\medskip

In the particular case where $\cal{G}_1 = \cal{G}_2$ the isomorphism $H$ is called a {\it graph-of-groups automorphism}. If furthermore the automorphism $H: \cal{G} \rightarrow \cal{G}$ induces the identity on the underlying graph $\Gamma(\cal{G})$, then the group isomorphisms for edge groups and vertex groups are all automorphisms. 

\begin{rem}
\label{long-formula}
(1)
Consider any group $G$ and any automorphism $ \Phi \in \Aut(G)$. As in \cite{Ye02}
(compare also \cite{R}), for any $g \in G$ and 
$s < t \in \Bbb{N}$,
we 
define 
{\it iterated products} 
$\Phi^{(t)}(g)$ and $\Phi^{(s,t)}(g)$, given by
$$\Phi^{(t)}(g) = g \Phi(g) \Phi^{2}(g) \cdots \Phi^{t-1}(g), $$
and
$$\Phi^{(s, t)}(g) =  \Phi^{s}(g) \Phi^{s+1}(g) \cdots \Phi^{t-1}(g).$$

\smallskip
\noindent
(2)
Note in particular that, in the case where $H: \cal{G} \rightarrow \cal{G}$ is a graph-of-groups isomorphism which induces the identity on the underlying graph $\Gamma(\cal{G})$, for any $t \in \Bbb{N}$ the iteration of $H_{*}$ on $t_{e}$ gives 
(via the well known formulae for the composition of graph-of-groups isomorphism, see Remark 2.10 of \cite{Ye01}):
\begin{align*}
H_{*}^{t}(t_e) 
& = H^{t-1}_{*}(\delta(\bar{e})) \ldots H_{*}(\delta(\bar{e}))\delta(\bar{e}) \cdot t_e \cdot \delta(e)^{-1} H_{*}(\delta(e)^{-1}) \ldots  H^{t-1}_{*}(\delta(e)^{-1}) \\
& = H^{(t)}_{*}(\delta(\bar{e})^{-1})^{-1}  \cdot t_e \cdot H^{(t)}_{*}(\delta(e)^{-1}). 
\end{align*}
\end{rem}




\subsection{Equivalences of graph-of-groups and their automorphisms}
\label{2.2}

${}^{}$

The following statements are well known:


\begin{lem}[Section~2.4 in \cite{Ye01}]
\label{gog-modif}
Let $\cal{G}$ and $\cal{G}^{\prime}$ be two graphs-of-groups which are identical everywhere, except that for some 
$e \in E(\cal{G})$ and $g \in G_{\tau(e)}$ one has $f^{\prime}_{e}=ad_{g} \circ f_{e}$.

Then there is a canonical graph-of-groups isomorphism $H_g: \cal G \to \cal G'$ which is the identity on the underlying graph, on all vertex and edge groups, and has all correction terms equal to $1$, except that $\delta(e) = g^{-1}$.
%
\qed
\end{lem}

\begin{lem}[
\cite{Ye01}, Lemma 2.11]
\label{isom-modif}
Let $\cal{G}$, $\cal{G}^{\prime}$ as well as $e, g$ and $H_g$ be as in Lemma \ref{gog-modif}, and let $H:\cal G \to \cal G$ be any graph-of-groups automorphism. 

Then there is a graph-of-groups automorphism $H': \cal G' \to \cal G'$ which coincides with $H$ everywhere except that 
$
\delta_{H'}(e) 
= H_{\tau(e)}(g) 
\delta_H(e) 
g^{-1}$, and $H'$ is conjugate to $H$ through $H_g$:
$${H}^{\prime}={H}_{g} \circ {H} \circ {H}_{g}^{-1}$$
\qed
\end{lem}

\begin{lem}[\cite{KLV01}, Corollary 4.8]
\label{auto-modif}
Let $H: \cal{G} \to \cal G$ and $H': \cal G \to \cal G$ be two graph-of-groups automorphisms, let $v$ be a vertex of $\cal G$ and $g$ an element of the vertex group $G_v$.  Assume that $H$ and $H'$ agree everywhere, except that 
$H'_v = ad_g H_v$
and for any edge $e$ with terminal vertex $\tau(e) = v$ one has $\delta_{H'}(e) = g \delta_H(e)$. Then $H$ and $H'$ induce the same outer automorphism:
$$\hat H = \hat H': \pi_1 \cal G \to \pi_1 \cal G$$
\qed
\end{lem}

We also need to consider how ``honest'' automorphisms, rather than outer ones, behave under base point change. This turns out to be a rather tricky issue:

\begin{rem}
\label{tricky}
${}^{}$
Let $H: \cal{G} \to \cal G$ be a graph-of-groups automorphism, and let $v$ and $v'$ be two vertices of $\cal G$
which are fixed by $H_{\Gamma}$.
For some word $U \in \Pi(\cal G)$, based on a path with initial vertex $v$ and terminal vertex $v'$, we consider the isomorphism 
$$\theta_{U}: \pi_1(\cal G, v') \to \pi_1(\cal G, v),\,\, W' \mapsto  W := U W' U^{-1} \,\, [= ad_{U}(W')]$$
and define $H_{*v', U} := ad_{U^{-1} H_{*}(U)} H_{*v'}$, in order to obtain, for a change of base point without changing the induced automorphism:
$$H_{*v', U} = \theta_{U}^{-1} H_{*v} \theta_{U}$$
%

Indeed, one has for any $W' \in \pi_1(\cal G, v')$ the equalities 
\begin{align*}
%
H_{* v', U}(W')  
= & \,\,\,  ad_{U^{-1} H_{*}(U)} H_{* v'}(W') \\
= & \,\,\,  U^{-1} H_{*}(U) H_{* v'}(W') H_{* }(U^{-1}) U \\
= & \,\,\,  H_{* v'}(H_*^{-1}(U^{-1}) U) H_{* v'}(W') H_{* v'}(U^{-1} H_*^{-1}(U)) \\
= & \,\,\,  H_{* }(H_*^{-1}(U^{-1})U W' U^{-1} H_*^{-1}(U))  \\
= & \,\,\,  U^{-1} H_{* v}(U W' U^{-1}) U  \\
= & \,\,\,  ad_{U^{-1}} H_{* v} ad_{U}(W')  \\
= & \,\,\,  \theta_U^{-1} H_{* v} \theta_U(W')
\end{align*}
\end{rem}




\subsection{Dehn twists and efficient Dehn twists}
\label{2.4}

${}^{}$

Dehn twists defined by means of graphs-of-groups seem to have appeared in the literature independently from each other at various places, with a slight degree of variation in generality and in the set-up. We follow here closely the ``original sources'' \cite{CL95, CL99, KLV01} and \cite{Ye01}; for an interesting alternative the reader is referred to \cite{Levitt}.


\begin{defn}
\label{Dehn-twist}
A {\it Dehn twist} $D: \cal{G} \rightarrow \cal{G}$ 
is a graph-of-groups automorphism such that the graph isomorphism $D_{\Gamma}$ as well as the group automorphisms $D_e$ and $D_v$, for any $e \in E(\cal{G})$ and any $v \in V(\cal{G})$, are all equal to the identity map. 
In addition, for any $e \in E(\cal{G})$ the correction term $\delta(e) \in G_{\tau(e)}$ is contained in the 
centralizer $C_e$ of $f_e(G_e)$ in $G_{\tau(e)}$. 

If 
one has 
$C_e = f_e(G_e)$ 
and $G_e$
is free, as is the case if $\pi_1 \cal G \cong \FN$ and $G_e \neq \{1\}$, then there is an element $\gamma_e$ in the center $Z(G_e)$ of $G_e$ such that $\delta(e)=f_e(\gamma_e)$.
In this case 
the {\it twistor} 
$z_e$ 
of $e$ is defined by $z_e = \gamma_{\bar{e}} \gamma_{e}^{-1}$. This yields $z_{\bar{e}} = z_{e}^{-1}$,
and for $z_e \neq 1$ it follows that $G_e \cong \Z$.
\end{defn}

For the rest of this subsection we will assume for simplicity that $D: \cal G \to \cal G$ is a {\em classical} Dehn twist, which means that all edge groups are infinite cyclic, so that 
the outer automorphism $\hat D$ induced by 
$D$ is well defined by specifying the twistor $z_e$ of every edge $e$.

\smallskip

Indeed, 
the Dehn twist $D$ determines an automorphism $D_*$ on the path group $\Pi(\cal{G})$ which given on the generators as follows:
\begin{enumerate}
\item[] $D_{*}(g) = g$, for $g \in G_{v}$, $v \in V(\cal{G})$;
\item[] $D_{*}(t_e) = t_e f_{e}(z_e)$, for $e \in E(\cal{G})$.
\end{enumerate}

The induced outer automorphism $\hat{D}$ of $\pi_1(\cal{G})$, as well as for any $v \in V(\cal G)$ the induced automorphism ${D}_{*v}$ of $\pi_1(\cal{G}, v)$
are called a {\it Dehn twist automorphism}. 
An automorphism $\phi \in \Out(\FN)$ is a {\it Dehn twist automorphism} if it is 
induced by some Dehn twist $D: \cal G \to \cal G$ via a suitable identification $\FN \cong \pi_1 \cal G$.

\begin{rem}
\label{preferred-lifts}
${}^{}$
Not every 
lift 
$\Phi \in \Aut(\FN)$ of a Dehn twist automorphism $\phi \in \Out(\FN)$ is itself 
represented by 
a Dehn twist automorphism, if we stick to a given identification $\FN \cong \pi_1 \cal G$, as there might not be a suitable vertex $v$ in $\cal G$.
%
\end{rem}

\begin{defn}

Given a Dehn twist $D: \cal{G} \rightarrow \cal{G}$, 
with 
family of twistors $(z_e)_{e \in E(\cal{G})}$,
any 
two edges $e_1$ and $e_2$ with common terminal vertex $v$ are said to be
\begin{enumerate}
\item {\it {positively bonded}}, if 
$f_{e_1}(z_{1}^{n_1})$ and $f_{e_2}(z_{2}^{n_2})$ are conjugate in $G_v$ for some
$n_1,n_2 \geq 1$.
\item {\it {negatively bonded}}, if 
$f_{e_1}(z_{1}^{n_1})$ and $f_{e_2}(z_{2}^{n_2})$ are conjugate in $G_v$ for some
$n_1 \geq 1$ 
and 
$n_2 \leq -1$.
\end{enumerate}
\end{defn}

\medskip

For the rest of this subsection, we always assume that $\cal{G}$ is a graph-of-groups such that its fundamental group $\pi_1(\cal{G})$ is free and of rank $N \geq 2$.

\begin{defn}
A Dehn twist $D: \cal{G} \rightarrow \cal{G}$ is said to be {\it efficient} if the following conditions are satisfied:
\begin{enumerate}
\item $\cal{G}$ is {\it minimal}: if $v = \tau(e)$ is a valence-one vertex, then the edge homomorphism $f_{e}: G_{e}\rightarrow G_{v}$ is not surjective.
\item There is no {\it invisible vertex}: there is no valence-two vertex $v=\tau(e_{1})=\tau(e_{2})$ $(e_{1} \neq e_{2})$ such that both edge maps $f_{e_{i}}: G_{e_{i}} \rightarrow G_{v}$ $(i=1,2)$ are surjective. 
\item No {\it proper power}: if $r^{p} \in f_{e}(G_{e})$ $(p \neq 0)$ then $r \in f_{e}(G_{e})$, for all $e \in E(\Gamma)$.
\item If $v= \tau(e_{1})=\tau(e_{2})$, then $e_{1}$ and $e_{2}$ are not positively bonded.
%
\item[(5)] No {\it unused edge}: for every $e \in E(\Gamma)$ the twistor satisfies $z_{e} \neq 1_{G_e}$ (or equivalently $\gamma_{e} \neq \gamma_{\bar{e}}$).
\end{enumerate}
\end{defn}

It has been shown in \cite{CL99} that every Dehn twist can be transformed algorithmically into an efficient Dehn twist. Thus every Dehn twist automorphism can be represented by some efficient Dehn twist.

Efficient Dehn twists are useful because of the following uniqueness result:

\begin{thm} [\cite{CL99}, Theorem 1.1]
\label{uniqueness}
Two efficient Dehn twists $D: \cal G \to \cal G$ and $D': \cal G' \to \cal G'$ define outer automorphisms 
that are conjugate to each other 
if and only if there is a graph-of-groups isomorphism $H: \cal G \to \cal G'$ with
$$\hat D' = \hat H \hat D \hat H^{-1} \, .$$ 
\end{thm}

\section{$H$-conjugation and 2-level graph-of-groups}

\begin{rem}
\label{Rod1}
Some of the material presented in this section seems to be in close proximity to work done by M. Rodenhausen in his thesis \cite{R}, in particular to his sections 4.4, 4.5, and 4.6.
\end{rem}

\subsection{$H$-conjugacy}
\label{2.3}

${}^{}$

The following definition, applied to graph-of-groups automorphisms, turns out to play a crucial role in our context:

\begin{defn}
\label{twisted-conjugacy}
Let $G$ be a group and $\Phi: G \to G$ be an automorphism of $G$. Then two elements $g, g' \in G$ are {\em $\Phi$-conjugate} to each other, written $g \simeq_\Phi g'$,  if there exists $h \in G$ such that
$$g' = h^{-1} g \, \Phi(h) \, .$$
The set of all elements $\Phi$-conjugate to $g$ will be denoted by $[g]_\Phi$ (where it easy to verify that $\simeq_\Phi$ is an equivalence relation and hence $[g]_\Phi$ a coset of the latter).

An element $g \in G$ is called $\Phi$-trivial if it is $\Phi$-conjugate to the neutral element $1 \in G$.
\end{defn}

It follows directly from this definition that $g \simeq_\Phi h$ if and only if 
$g^{-1} \simeq_{\Phi^{-1}} h^{-1}$.
However, note that 
$\Phi$-conjugacy and $\Phi^{-1}$-conjugacy do in general disagree.

\medskip

We will now specialize to the case where $G$ and $\Phi$ are given in graph-of-groups language:

Let $\cal G$ be a graph-of-groups, let $v$ be a vertex of $\cal G$, and let $H: \cal G \to \cal G$ be a graph-of-groups automorphism, which 
throughout most of
this subsection 
will 
act trivially on the underlying graph $\Gamma(\cal G)$. Then $H$ induces an automorphism $H_{*v}: \pi_1(\cal G, v) \to \pi_1(\cal G, v)$, and the notions introduced in Definition \ref{twisted-conjugacy} can be applied to $H_{*v}$. However, the group $\pi_1(\cal G, v)$ is canonically embedded in the ambient group $\Pi(\cal G)$, and many issues, in particular those coming from a change of base point in $\cal G$, can be much better understood there. 
It might thus be tempting to pass directly to the automorphism $H_*: \Pi(\cal G \to \Pi(\cal G)$, which after all restricts on the subgroup $\pi_1(\cal G, v)$ to $H_{*v}$, and to consider directly $H_*$-conjugacy in $\Pi(\cal G)$. This, however, would lead to rather undesired phenomena:

\begin{example}
\label{Kaidis-example}
${}^{}$
Let $H: \cal G \to \cal G$ be a 
graph-of-groups
isomorphism 
which acts as identity on $\Gamma(\cal G)$ and on all vertex groups. We furthermore assume that $\cal G$ has trivial edge groups.
We specify $\cal G$ and $H$ as follows:

\smallskip
\noindent
(1) 
Let $\Gamma(\cal G)$ be the graph which consist of a single edge $e$ and two distinct vertices $v_1 = \tau(e) \neq v = \tau(\bar e)$, and set $\delta(e) = a^{-1} \neq 1 \in G_{v_1}, \, \delta(\bar e) = b^{-1} \neq 1 \in G_{v}$. For $t = t_e$ this gives $H_*(t) = b^{-1} t a$.

We compute $a \simeq_{H_*} t a (a^{-1} t^{-1} b) = b$, thus obtaining an example of two non-trivial reduced words $a$ and $b$ in $\Pi(\cal G)$ which are ${H_*}$-conjugate to each other, but their underlying loops are distinct (since trivial loops at distinct vertices are distinct). 

\smallskip
\noindent
(2) 
We start out with all data as in 
example (1), 
except that $a = 1$. We then add a second edge $e'$ and a third vertex $v_2$ different from $v$ and $v_1$, with $v_2 = \tau(e')$ and $v = \tau(\bar e')$. We set $\delta(e') = 1 \in G_{v_2}$ and $\delta(\bar e') = c^{-1} \in G_{v}$, so that for $t' = t_{e'}$ one has $H_*(t') = c^{-1} t$.

As in 
example (1) 
we compute $1 \simeq_{H_*} t 1 (t^{-1} b) = b$ and  $1 \simeq_{H_*} t' 1 (t'^{-1} c) = c$, and thus $b \simeq_{H_*} c$. Thus, if we further specify $G_{v_1} = G_{v_2} = \{1\}$ and $G_v \cong F_2 = \langle b, c \rangle$, then we see that $H_{*v} = id_{F_2}$, but the two generators $b$ and $c$ are $H_*$-conjugate. 

Alternatively one could set $G_v \cong \Z = \langle b \rangle$ and $c = b^2$, thus getting $b \simeq_{H_*} b^2$ for $H_* = id_\Z$.
\end{example}

The reason for the ``misbehavior'' in the above Example \ref{Kaidis-example} (2) comes from the fact that $H_{*v}$-conjugacy classes do not inject into $H_*$-conjugacy classes (under the subgroup embedding): the $H_*$-trivial class has more than one preimage. We remedy this by the following definition, where we note that in the definition of a ``reduced word'' $W$ (see Definition \ref{reduced-word}) the case $W = 1 \in \Pi(\cal G)$ is excluded.

\begin{defn}
\label{trivial-reduced}
(1)
Let $W^*(\cal G)$ be defined as $W(\cal G)$, except that, whenever the trivial element $1 \in \underset{v \in V(\Gamma)}{\operatorname{\ast}}G_{v}$ appears in a reduced word $W$ (possibly invisible, 
since suppressed when it occurs between two subsequent letters from $\{t_{e} \mid e\in E(\Gamma)\}$) then it has to be specified to which of the free factors $G_v$ it belongs. We 
do this notationally by writing $1_v := 1_{G_v}$.

\smallskip
\noindent
(2)
Furthermore, if a 
word $W \in W^*(\cal G)$ is {\em connected}, then we require in addition that for any syllabe $t_e 1_v t_{e'}$ of $W$ one has $v = \tau(e) = \tau(\bar e')$.
In this case, however, since $v$ is uniquely determined by $e$ and $e'$, one is again allowed to suppress $1_v$ and write again $t_e t_{e'}$ for $t_e 1_v t_{e'}$.

\smallskip
\noindent
(3)
Accordingly, we say that 
a connected word $W \in \Pi^*(\cal G)$
is a {\em reduced} 
if either 
$W \in \Pi(\cal G) \smallsetminus \{1\}$
and $W$ is reduced, or else if $W = 1_v$ for some vertex $v$ of $\cal G$.
We carry over the terminology from Definition \ref{reduced-word} (1) 
and
(2) in the obvious way, where for $W= 1_v$ the {\em terminal vertex} $\tau(W)$ is given by $v$, and $W$ is {\em issued at} $v$.
\end{defn}

We'd like to give the reader a hint why the slightly bizarre definition of $\Pi^*(\cal G)$ makes sense: What one really considers here are pairs $(W, v)$, where $W$ is an element of $\Pi(\cal G)$ with underlying connected path $\gamma$, given as reduced connected word, and $v$ is its terminal vertex in $\cal G$. However, since for $|W| \geq 1$ the underlying non-trivial path $\gamma$ determines $v =\tau(\gamma)$, writing $(W, v)$ for $W$ would just be uselessly cumbersome. The same reasoning extends to $|W| = 0$ as long as $W \in G_v \smallsetminus \{1\} \subset \Pi(\cal G)$. But for any pair $(1, v)$ the second coordinate becomes meaningful, and thus the reader is free to simply interpret $1_v$ as an abbreviation for the pair $(1, v)$.

\begin{rem}
\label{extension-action}
It is easy to see that with the above definitions the $H_*$-action on $\Pi(\cal G)$ extends directly to a well defined $H_*$-action on $\Pi^*(\cal G)$, through setting $H_*(1_v) := 1_{H_\Gamma(v)}$. 
\end{rem}

We can now go on to define:

\begin{defn}
\label{H-conjugation}
Let $H: \cal G \to \cal G$ be a graph-of-groups automorphism which acts trivially on $\Gamma(\cal G)$.
Let $W, W' \in \Pi^*(\cal G)$ be two closed connected words, issuing from vertices $v$ and $v'$ respectively. Then $W$ and $W'$ are {\em $H$-conjugate}, written $W \simeq_H W'$, if there exists a connected word $U$ with initial vertex $v$ and terminal vertex $v'$ such that one has: 
$$W' = U^{-1} W H_*(U)$$
\end{defn}

\begin{rem}
\label{vertex-conjugacy}
(1)
It is easy to verify 
that $H$-conjugation is indeed an equivalence relation, 
where we denote the equivalence class of $W$ by $[W]_H$.

\smallskip
\noindent
(2)
It follows directly, for closed connected words $W$ and $W'$ 
issuing from the same vertex $v$, that $W$ and $W'$ are $H$-conjugate if and only if they are $H_{*v}$-conjugate. 

\smallskip
\noindent
(3)
Furthermore, for any two vertices $v$ and $v'$ and any ``connecting word'' $U$ as in Definition \ref{H-conjugation}, if we use conjugation by $U$ in $\Pi(\cal G)$ to identify $\pi_1(\cal G, v)$ and $\pi_1(\cal G, v')$, then $H_{*v}$-conjugacy and $H_{*v', U}$-conjugacy coincide (for $H_{*v', U}$ as defined in Remark \ref{tricky}).

\smallskip
\noindent
(4)
However, it is important to note that in the situation considered in (3) above, the notion of being ``$H_{*v}$-trivial'' and ``$H_{*v'}$-trivial'' do {\em not} coincide: In general, $H$-conjugation of an element $W \simeq_{H_{*v}} 1$ 
in $\pi_1(\cal G, v)$
will give an element 
$W' = U^{-1} W H_*(U) \in \pi_1(\cal G, v')$
which is not $H_{*v'}$-trivial.
\end{rem}

The notion of $H$-conjugacy enables us to perform {\em $H$-reduction} on a closed connected word $W \in \Pi(\cal G)$, by $H$-conjugating it to a word $W'$ with $|W'|_\cal G < |W|_\cal G$. We say that $W$ is {\em $H$-reduced} if such a shortening of the length through $H$-conjugation is not possible.

\begin{rem} [\cite{Ye01}, 
Remark 4.8]
\label{H-reduced}
${}^{}$
(1)
Let $W, W' \in \Pi^*(\cal G)$ be two $H$-reduced closed connected words, with underlying 
closed paths $\gamma$ and $\gamma'$ respectively. If $W$ is $H$-conjugate to $W'$, 
and either $\gamma$ or $\gamma'$ is non-trivial,
then it can be shown that
$\gamma$ and $\gamma'$ must agree up to a cyclic permutation.

However, if both $\gamma$ and $\gamma'$ are trivial, this conclusion may fail, as shown in Example \ref{Kaidis-example} (1).


\smallskip
\noindent
(2)
We say that a closed connected word $W \in \Pi^*(\cal G)$ is {\em $H$-zero} if it is $H$-conjugate to some word $W'$ of $\cal G$-length $|W'|_\cal G = 0$.

In other words, after $H$-reducing $W$ we obtain a word $W'$ which is based on a trivial loop. In general, however, the information on which vertex this trivial loop starts and finishes depends on $W'$ (and thus on $W$) and can not be changed without changing the $H$-conjugacy class.
\end{rem}


\subsection{2-level graph-of-groups}
\label{2.3.5}

${}^{}$

In 
the subsequent sections of 
this paper we will consider the situation where a graph-of-groups $\cal G$ is given by defining, for each vertex $V$ of $\cal G$, the vertex group $G_V$ through a (possibly 
trivial) {\em local} graph-of-groups $\cal G_V$:
$$G_V \cong \pi_1 \cal G_V$$
To avoid confusion, we denote in this context the non-local vertices and edges by capital letters. For simplicity we restrict our attention to the only case which matters for this paper, namely where all 
edge groups $G_E$ of $\cal G$ are trivial:
$$G_E = \{1\}$$

We now want to define a graph-of-groups automorphism $H$ of $\cal G$, and we'd like to do this while minimizing the amount of technical data given
though making sure that the outer automorphism $\hat H$ is well defined. 
Again for simplicity, we assume that the underlying graph automorphism $H_\Gamma$ equals to the identity map $id_{\Gamma(\cal G)}$.

In order to define such $H: \cal G \to \cal G$, we need the following data:
\begin{enumerate}
\item
For every vertex $V$ of $\cal G$ we consider a {\em local} graph-of-groups automorphism $\cal H_V: \cal G_V \to \cal G_V$, for which we also assume $(\cal H_V)_{\Gamma} = id_{\Gamma(\cal G_V)}$.
\item
For every edge $E$ of $\cal G$ with terminal vertex $V = \tau(E)$ let $\delta^*(E) \in \Pi^*(\cal G_V)$ be a closed connected word.
\item
Furthermore, for every vertex $V$ of $\cal G$ we choose a {\em local base point} $v_V$ in $\cal G_V$, and for every edge $E$ with $\tau(E) = V$ a {\em connecting word} $U_E \in \Pi^*(\cal G_V)$, with underlying 
path $\gamma_E$ which connects 
$\tau(\delta^*(E))$ to $v_V$.
\end{enumerate}

We now specify the vertex groups $G_V = \pi_1(\cal G_V, v_V)$ and the correction terms $\delta(E) = (\cal H_V)_*(U_E^{-1}) \delta^*(E) U_E$ and obtain thus a well defined graph-of-groups automorphism $H: \cal G \to \cal G$ with underlying graph automorphism $H_\Gamma = id_{\Gamma(\cal G)}$, and with vertex group automorphisms given by 
$H_V = (\cal H_V)_{* v_V}$
for each vertex $V$ of $\cal G$. The equality 
(\ref{gog-isom-eq})
from Definition \ref{gog-iso} 
is automatically  satisfied, since all edge groups $G_E$ are trivial.

\begin{prop}
\label{independence}
The outer automorphism $\hat H$ on $\pi_1 \cal G$ induced by 
$H$ depends only on the above data (1) and (2), but not on (3).  

More precisely, if $H': \cal G' \to \cal G'$ is a second graph-of-groups automorphisms which agrees with $H$ in (1) and (2), then there is a graph-of-groups isomorphism $H_0: \cal G \to \cal G'$ such that
one has:
$$\hat H' = \hat H_0^{-1} \hat H \hat H_0$$
\end{prop}

\begin{proof}
For any vertex $V$ of $\Gamma = \Gamma(\cal G) = \Gamma(\cal G')$ let $v_V$ and $v'_V$ be the local base points in $\cal G_V$ for $\cal G$ and $\cal G'$ respectively, given by condition (3) above. Similarly, for any $E$ of $\Gamma$ we denote 
the connecting words from (3) by $U_E$ and $U'_E$. The issuing correction terms for $H$ and $H'$ are given by $\delta(E) = (\cal H_V)_*(U_E^{-1}) \delta^*(E) U_E$ and $\delta'(E) = {(\cal H_V)}_*({U'_E}^{-1}) \delta^*(E) U'_E$ respectively.
The conjugating graph-of-groups isomorphism $H_0$ is now constructed in three steps:

\smallskip
\noindent
(1) 
First, choose for every vertex $V$ a word $U_V \in \Pi^*(\cal G_V)$ with underlying path that connects $v_V$ to $v'_V$, and consider the isomorphisms 
$$\theta_{U_V}: \pi_1(\cal G_V, v'_V) \to \pi_1(\cal G_V, v_V), \, W \mapsto U_V W U_V^{-1}$$
and 
\begin{align*}
(\cal H_V)_{*v',U_V} &= \theta_{U_V}^{-1} (\cal H_V)_{*v} \theta_{U_V} \\
&= ad_{U_V^{-1} (\cal H_V)_{*}(U_V)} \circ (\cal H_V)_{*v'}: \pi_1(\cal G_V, v') \to \pi_1(\cal G_V, v')
\end{align*}
from Remark \ref{tricky}.
One then defines a 
graph-of-groups
isomorphism $F: \cal G \to \cal G'$ 
with $F_\Gamma = id_{\Gamma(\cal G)}$ 
by setting $F_V = \theta_{U_V}^{-1}$ for every vertex $V$ 
and $\delta_F(E) = 1$ for every edge $E$ 
of $\cal G$, and obtains $H'' = F H F^{-1}: \cal G' \to \cal G'$ with $H''_{V} := (\cal H_V)_{*v',U_V}$ and 
(using the formulae from Remark 2.10 of \cite{Ye01})
$$\delta_{H''}(E) = \theta_{U_V}^{-1}(\delta(E)) = U_V^{-1} (\cal H_V)_*(U_E^{-1}) \delta^*(E) U_E U_V$$
for every edge $E$ with terminal vertex $V$.

\smallskip
\noindent
(2) 
We now use Lemma \ref{auto-modif}
to define a graph-of-groups automorphism $H''': \cal G' \to \cal G'$ with $\hat{H}''' = \hat{H}''$, such that for every vertex $V$ of $\cal G'$ one has $H'''_V = ad^{-1}_{U_V^{-1} (\cal H_V)_{*}(U_V)} H''_V = (\cal H_V)_{*v'}$ and for every edge $E$ of $\cal G'$ one has:
\begin{align*}
\delta_{H'''}(E) &= (U_V^{-1} (\cal H_V)_{*}(U_V))^{-1} \delta_{H''}(E) \\
&= [(\cal H_V)_{*}(U_V)^{-1} U_V] [U_V^{-1} (\cal H_V)_*(U_E^{-1}) \delta^*(E) U_E U_V]\\
&= (\cal H_V)_{*}(U_V)^{-1}  (\cal H_V)_*(U_E^{-1}) \delta^*(E) U_E U_V) \\
&= (\cal H_V)_{*}(U_V^{-1}  U_E^{-1}) \delta^*(E) U_E U_V)
\end{align*}

\smallskip
\noindent
(3) 
Finally, one applies Lemma \ref{isom-modif}
to obtain a graph-of-groups automorphism $R': \cal G' \to \cal G'$ which is the identity everywhere and conjugates $H'''$ to $H'''' = R' H''' R'^{-1}$, such that the correction term $\delta_{H''''}(E)$ of any edge $E$ can take on an arbitrary value within its $(H'''_V)^{-1}$-equivalence class. In particular one can choose $R'$ to obtain 
\begin{align*}
\delta_{H''''}(E)
&= H'''_V((U_V^{-1} U_E^{-1} U'_E)^{-1}\cdot \delta_{H'''}(E) \cdot (U_V^{-1} U_E^{-1} U'_E)\\
&= 
(\cal H_V)_{*v'}(U_V^{-1} U_E^{-1} U'_E)^{-1}[(\cal H_V)_{*}(U_V^{-1}  U_E^{-1}) \delta^*(E) U_E U_V)]  U_V^{-1} U_E^{-1} U'_E\\
&= {(\cal H_V)}_*({U'_E}^{-1}) \delta^*(E) U'_E = \delta'(E)
\end{align*}
and hence
$H'''' = H'$, which proves our claim.
\end{proof}

In the above proof we have actually shown something slightly stronger:

\begin{cor}
\label{class-dependence}
The outer automorphism $\hat H$, for $H$ as in Proposition \ref{independence}, depends only on the 
$\cal H_{\tau(E)}^{-1}$-conjugacy classes of the elements $\delta^*(E)$ and not on the 
representatives
$\delta^*(E)$ themselves.
\qed
\end{cor}


The following definition turns out to be very useful:

\begin{defn}
\label{locally-zero}
We say that an edge $E$ of $\cal G$, assumed as before to have trivial edge group $G_E = \{1\}$ and terminal vertex $V$, is {\em locally zero} if the correction term $\delta(E) \in G_V \cong \pi_1 \cal G_V$ 
is given by an element $\delta^*(E) \in \Pi^*(\cal G_V)$ which is 
$\cal H_V^{-1}$-zero (i.e. 
$\cal H_V^{-1}$-conjugate to an element of $\cal G_V$-length 0,
see Remark \ref{H-reduced} (2)).
\end{defn}


Let us finish this subsection by pointing out that it is fairly easy to give examples of automorphisms $\phi, \phi' \in \Out(\FN)$ which are defined by graph-of-groups isomorphisms $H$ and $H'$ as in Proposition \ref{independence} respectively, where all edges are locally zero, with correction terms given by elements that are indeed ``locally trivial'' (i.e. any $\delta^*(E)$ is 
$\cal H_{\tau(E)}^{-1}$-conjugate to the trivial element), but $\phi$ and $\phi'$ are not conjugate in $\Out(\FN)$, since for some edge $E$ the element $\delta^*(E)$ is 
$\cal H_{\tau(E)}^{-1}$-conjugate to $1_v$, while ${\delta'}^*(E)$ is 
${\cal H'}_{\tau(E)}^{-1}$-conjugate to $1_{v'}$ with $v \neq v'$.


\section{2-level Dehn twists}
\label{sec-3}

From Definition \ref{Dehn-twist} we see that any graph-of-groups automorphism $D: \cal G \to \cal G$, which is the given through the identity map on the underlying graph $\Gamma(\cal G)$, 
on any vertex group $G_V$ and on any edge group $G_e$, is a Dehn twist. Inspired by this observation, in \cite{Ye01} {\em partial Dehn twists} $H: \cal G \to \cal G$ have been defined, which differ from $D$ as before in that for some vertices of $\cal G$ the vertex group automorphisms are different from the identity map. A special case, already considered in \cite{Ye01}, is given by the following:

\begin{defn}
\label{2-level-D-twist}
A {\em 2-level Dehn twist} is given by a graph-of-groups $\cal G$ with trivial edge groups, and a graph-of-groups isomorphism $H: \cal G \to \cal G$ which induces the identity on the underlying graph $\Gamma(\cal G)$ and which induces on every (possibly trivial) vertex group a Dehn twist automorphism or an inner automorphism.
\end{defn}

To be more specific regarding the vertex group automorphisms, we want to use now the concept of a 2-level graph-of-groups as introduced in section \ref{2.3.5} to describe a 2-level Dehn twist. As done there, we will denote vertices and edges of $\cal G$ by capital letters, while for all local graph-of-groups we use small letters. We first note a slightly technical point,
which arises when $H$ as in Definition \ref{2-level-D-twist} is specified to an automorphism of a 2-level graph-of-groups as in section \ref{2.3.5}:

\begin{rem}
\label{Dehn-auto}
In Definition \ref{2-level-D-twist}, at any vertex $V$ of $\cal G$ the local Dehn twist automorphism at $V$ is formally defined as element 
of 
$\Aut(G_V)$ and not 
of 
$\Out(G_V)$.  From Remark \ref{preferred-lifts} it may appear that this is a 
serious 
restriction,
once the identification $G_V \cong \pi_1 \cal G_V$ is fixed. However, it follows from Lemma \ref{auto-modif}
that this is not true, 
since one can adapt the correction terms of any edge $E$ with terminal vertex $V$ accordingly.
\end{rem}

As has been recalled in 
section 
\ref{2.4},
every Dehn twist automorphism of a free group $\FN$ can be represented by an {\em efficient} Dehn twist $D: \cal G \to \cal G$, 
and the 
latter 
is unique 
(on the level of the induced outer automorphisms) 
up to conjugation with graph-of-groups isomorphisms, see Theorem \ref{uniqueness}. We will hence assume below always that 
any 
2-level Dehn twist $H: \cal G \to \cal G$ comes for every vertex group $G_V$ of $\cal G$ with an efficient Dehn twist $D_V: \cal G_V \to \cal G_V$ and an identification $G_V \cong 
\pi_1 \cal G_V$
such that $D_V$ induces the outer automorphism defined by the vertex group automorphism $H_V: G_V \to G_V$.

\begin{rem}
\label{clarification}
(1)
In the above set-up, the suppression of a local base point $v_V \in V(\cal G_V)$ in the identification $G_V \cong \pi_1 \cal G_V$ (rather than specifying $G_V \cong \pi_1 (\cal G_V, v_V)$), and 
naming only 
the condition $\hat D_V = \hat H_V$ (rather than specifying it to  $D_V = H_{*v_V}$), is not notational sloppiness, but rather has been done purposefully.

Indeed,
we recall from Proposition \ref{independence} that for any edge $E$ of $\cal G$, say with terminal vertex $V =\tau (E)$, in order to determine the outer automorphism $\hat H$ induced by $H$, it suffices to specify a word $\delta^*(E) \in \Pi^*(\cal G_V)$. Indeed, the $D_V^{-1}$-equivalence class of $\delta^*(E)$ is sufficient, see Corollary \ref{class-dependence}, so that we will notationally not distinguish between the word $\delta^*(E)$ and the element $\delta(E)$ used in Proposition \ref{independence} to specify $H$, and simply refer to either as the ``correction term'' of $E$. Whenever the only case occurs where the difference between $\Pi^*(\cal G_V)$ and $\Pi(\cal G_V)$ matters, i.e. if $\delta(E) = 1$, we will be careful and specify $1$ to $1_v$ for 
the appropriate 
vertex $v$ of $\cal G_V$.

\smallskip
\noindent
(2)
We also recall (see Definition \ref{locally-zero}) that for 
the edge $E$ the notion of being ``locally zero'' is well defined, without having specified the base point $v_V$ for the graph-of-groups $\cal G_{V}$.
\end{rem}

In the following definition the fact that $\pi_1 \cal G$ is a free group is only assumed by practical reasons for this paper; the notion of efficient 2-level Dehn twists makes also sense for more general groups.

\begin{defn}
\label{2-level-efficient}
A 2-level Dehn twist $H: \cal G \to \cal G$  is called {\em efficient} if $\pi_1 \cal G$ is free of finite rank 
$N \geq 2$ 
and if the following conditions are satisfied:
\begin{enumerate}
\item
For any edge $E$ of $\cal G$ precisely one of the two, $E$ or $\bar E$, is locally zero
(i.e. 
precisely one of the two correction terms, $\delta(E) \in G_{\tau(E)}$ or $\delta(\bar E) \in G_{\tau(\bar E)}$, 
is $\cal H_V^{-1}$-conjugate to an element of $\cal G_V$-length 0, 
for $V = \tau(E)$ or $V = \tau(\bar E)$ respectively).

We say that $E$ is {\em forward oriented} if 
$E$
is not locally zero, and we define $E^+(\Gamma(\cal G)) = E^+(\cal G)$  to be the orientation on $\Gamma(\cal G)$ 
which contains all such forward oriented edges.
\item
For any two 
distinct forward oriented  
edges $E$ and $E'$ of $\cal G$ with common terminal vertex $V := \tau(E) = \tau(E')$, the correction terms $\delta(E)$ and $\delta(E')$ are not 
$D_V^{-1}$-conjugate.
\end{enumerate}
\end{defn}



Given any 2-level Dehn twist $H_0: \cal G_0 \to \cal G_0$, we can iteratively transform $\cal G_0$ and $H_0$ through intermediate 2-level Dehn twists $H_1: \cal G_1 \to \cal G_1$, $H_2: \cal G_2 \to \cal G_2$, etc, with canonical isomorphisms $\pi_1 \cal G_j \cong \pi_1 \cal G_{j+1}$ that induce $\hat H_j = \hat H_{j+1}$, to obtain after finitely many steps a 2-level Dehn twist $H_m: \cal G_m \to \cal G_m$ which is efficient. 

The modifications, employed in this procedure to pass from $H_j: \cal G_j \to \cal G_j$ to $H_{j+1}: \cal G_{j+1} \to \cal G_{j+1}$, are all of one of the following four types:

\begin{enumerate}
\item
Subdivide an edge $E$ by introducing a new vertex with trivial vertex group. 
Choose the correction terms on the subdivided edges to be trivial at the new vertex, and to coincide with $\delta(E)$ or $\delta(\bar E)$ otherwise.

This subdivision is in particular always done if both $E$ and $\bar E$ are not locally zero.

\item
Contract an edge $E$, if both $E$ and $\bar E$ are locally zero, through a blow-up of the local graph-of-groups automorphisms $D_{\tau(E)}$ and $D_{\tau (\bar E)}$ along $E$ as introduced in \cite{Ye01}. Subsequently make the resulting local Dehn twist on the new blown-up vertex group again efficient.

\item
For any edge $E$ with with terminal vertex $V$ and  
correction term $\delta(E)$ 
which is 
not
$D_V^{-1}$-zero, 
if there is any second edge $E' \neq E$, also with terminal vertex $V$, such that $\delta(E)$ is 
$D_V^{-1}$-conjugate to $\delta(E')$, one performs 
$D_V^{-1}$-conjugation on $\delta(E)$ to obtain $\delta(E) = \delta(E')$. 

\item
If two 
distinct 
edges $E$ and $E'$ with 
common 
terminal vertex $V$ have equal correction 
term 
$\delta(E) = \delta(E')$ which is not 
$D_V^{-1}$-zero, 
we first perform a subdivision of $E$ and $E'$ as in modification (1) above, so that 
we can now assume that $\delta(\bar E) = \delta(\bar E') = 1$
and $V_1 := \tau(\bar E) \neq V_2 := \tau(\bar E')$. 

We then fold $E$ onto $E'$ and identify $V_1$ and $V_2$ to get a new vertex $V'$ with trivial vertex group.
\end{enumerate}

Any such modification does not increase the number of edges $E$ which are not locally zero. However, in the process of doing our modifications, one eventually decreases their number (through modifications (3) and (4)) until any two edges $E$ and $E'$ with common terminal vertex $V$ have correction terms in distinct 
$D_V^{-1}$-conjugacy classes. We then finish the procedure by applying iteratively the modification (2) finitely many times. This shows:

\begin{prop}
\label{existence}
${}^{}$
For every 2-level Dehn twist with fundamental group $\FN$ there exists an efficient 2-level Dehn twist which
defines the same outer automorphism of $\FN$.
\end{prop}

In \cite{LY2} we will give more details which will show that the 
procedure described here is in fact algorithmic.

\begin{rem}
\label{Rod2}
Moritz Rodenhausen has obtained in his thesis \cite{R} results that seem to come very close to what has been presented in this section, in a context that is more general (polynomial growth automorphisms of arbitrary degree), and in a language that is not far from ours, but in its technical details sufficiently different to make a formal translation non-evident. In any case, it seems more than likely that the algorithmic device in his section 8 is strongly related to the above explained procedure, and in particular his Theorem 8.6 to our Proposition \ref{existence}.
\end{rem}


\section{Cancellation bounds}

\subsection{Some basic cancellation facts on graph-of-groups}

${}^{}$

Recall that 
for any graph-of-groups $\cal G$ and any reduced word 
$$W = w_0 t_1 w_1 \ldots w_{r-1} t_r w_r \in \Pi(\cal G)$$
we denote by $|W| := r$ the {\em $\cal G$-length} of $W$.

For any two reduced 
words $V, W \in \Pi(\cal G)$ we say that {\em $V$ cancels against $W$ in the product $V W$} 
if one has $|VW| = 0$.

Furthermore, we say that {\em the cancellation in a family of products $
W_1(t) \cdot W_2(t) \cdot \ldots \cdot W_m(t)$ is bounded}, for reduced words $W_1(t), $ $W_2(t), $ $\ldots, W_m(t) \in \Pi(\cal G)$, if there exists a constant $K \geq 0$ independent of the parameter $t$ such that:
$$|W_1(t)| + |W_2(t)| + \ldots + |W_m(t)| - |W_1(t) W_2(t) \ldots W_m(t)| \leq K$$
holds for any value of $t$.

\begin{rem}
\label{cancellation-isomorphism}
For any graph-of-groups isomorphism $H: \cal G \to \cal G'$ the above notions are preserved under $H$. For example, if $V$ cancels against $W$ in the product $V W$, then $H(V)$ cancels against $H(W)$ in the product $H(V W)$.
\end{rem}

\begin{defn}
\label{bondedness}
For any 
graph-of-groups 
$\cal G$ we say that edges $e$ and $e'$ are {\em bonded} if they terminate at the same vertex $v = \tau(e) = \tau(e')$, and if there exist non-trivial elements $g_e \in G_{e}$ and $g_{e'} \in G_{e'}$ with images $f_{e}(g_e)$ and $f_{e'}(g_{e'})$ that are conjugate in the vertex group $G_v$. 

To be specific, 
we call 
an element $h \in G_v$ with $f_{e}(g_e) = h f_{e'}(g_{e'}) h^{-1}$ 
a {\em bond conjugator}, while $g_e$ and $g_{e'}$ are called the {\em edge bonders}.

A path $e_1 e_2 \ldots e_r$ in $\cal G$ is a {\em bonded path} if for subsequent indices the edges $e_i$ and $\bar e_{i+1}$ are bonded, with a family of edge bonders that satisfy $g_{e_i} = g_{\bar e_i}$ for any $i = 2, \ldots, r-1$.

A {\em bonded loop} is analogously defined, with index $i$ understood cyclically modulo $r$.

Any trivial path or trivial loop is formally defined to be bonded.

\end{defn}

\begin{lem}
\label{unique-cancellation}
Let $\cal G$ be a graph-of-groups, and let $V, W \in \Pi(\cal G)$ be reduced words with common underlying path which is not bonded, and which terminates in some vertex $v$. Assume that $V$ cancels against $W^{-1}$ in the product $VW^{-1}$. Then for no $u \in G_v \smallsetminus \{1\}$ the element $Vu$ cancels against $W^{-1}$ in the product $VuW^{-1}$.
\end{lem}

\begin{proof}
Let $V = w_0 t_1 w_1 \ldots w_{r-1} t_r w_r$ 
be the given reduced word, based on a path $e_1 \ldots e_r$. From the assumption that $V$ cancels against $W^{-1}$ in the product $VW^{-1}$ we know that $W = 
u' V
= u' w_0 t_1 w_1 \ldots w_{r-1} t_r w_r$ for some $u' \in G_{\tau(\bar e_1)}$. Hence, if $Vu$ cancels against $W^{-1}$ in the product $VuW^{-1}$, then from the normal form for reduced words in $\Pi(\cal G)$ 
(see Remark \ref{gog-normal-form})
we know that for $k = 1, \ldots , r$ there exist elements $u_{k} \in G_{e} = G_{\bar e}$, such that $f_{e_k}(u_k) = w_k f_{\bar e_{k+1}}(u_{k+1}) w_k^{-1}$ for all $k = 1, \ldots , r-1$, and $f_{e_r}(u_r) = u$. 
From the assumption $u \neq 1$ it follows hence that $u_k \neq 1$ for all indices $k$.
But then the path $e_1 \ldots e_r$ is bonded, which contradicts our assumption.
\end{proof}

\begin{lem}
\label{non-bonded}
Let $D:\cal G \to \cal G$ be an efficient Dehn twist,
with $\pi_1 \cal G = \FN$ and $N \geq 2$. 
Then no 
non-trivial 
loop in $\cal G$ is bonded.
\end{lem}

\begin{proof}
This is a direct consequence of Definition \ref{bondedness} and the observation that a 
non-trivial 
bonded loop leads directly to commuting elements in $\pi_1\cal G$ which are not powers of each other, as one is based on the loop, and the other one is represented by the edge bonder, which can be ``shifted'' around the loop. 

However, this is impossible as by assumption $\pi_1 \cal G$ is a free group $\FN$ if rank $N \geq 2$.
\end{proof}

\begin{lem}
\label{D-image-cancellation}
Let $D: \cal G \to \cal G$ be an efficient Dehn twist, 
with $\pi_1 \cal G = \FN$ and $N \geq 2$, 
and let $V, W \in \Pi(\cal G)$ be reduced words 
based on 
non-trivial  
loops 
which are proper powers. 
If $V$ cancels against 
$W^{-1}$ in $VW^{-1}$, 
then for no $t \neq 0$ the element $V$ cancels against 
$D^t(W^{-1})$ in $VD^t(W^{-1})$.
\end{lem}

\begin{proof}
As in the above proof of Lemma \ref{unique-cancellation} we have reduced words $V = $ $w_0 t_1 w_1 \ldots w_{r-1} t_r w_r$ and $W = 
u' V  
= u' w_0 t_1 w_1 \ldots w_{r-1} t_r w_r$ for some $u' \in G_{\tau(\bar e_1)}$, both based on a path $e_1 \ldots e_r$ which is assumed to be closed. We compute $D^t(W) = 
u' w_0 t_1 f_1(z_1)^t w_1 \ldots w_{r-1} t_r f_r (z_r)^t w_r$, where we use the convention $t_k := t_{e_k}$ and $z_k := z_{e_k}$.

Thus, if $V$ cancels against $D^t(W^{-1})$ in $VD^t(W^{-1})$, then one obtains iteratively for all $k = 1, \ldots, r-1$ that $w_k f_{\bar e_{k+1}}(z_{k+1}^t) w_k^{-1}$ is equal to  $f_k(z_k)^{m_k}$, for some $m_k  \in t \Z$, since by the definition of an efficient Dehn twist both, the generators of the cyclic subgroups $f_{e_k}(G_{e_k})$ and $w_k f_{\bar e_{k+1}}(G_{e_{k+1}}) w_k^{-1}$ are not proper powers in their ambient vertex group $G_{\tau(e_k)}$. Furthermore we see iteratively that all $m_k$ must have the same sign as $t$, as otherwise $e_k$ and $e_{k+1}$ were positively bonded, contradicting the definition of an efficient Dehn twist. In particular the exponents can never add up to 0.

Thus any subsequent edges in the 
path 
$e_1 \ldots e_r$ are bonded.
But by the ``proper powers'' assumption on $e_1 \ldots e_r$ in the statement of our proposition, this path must run more than once around some non-trivial loop, which hence must also be bonded. But this
contradict Lemma \ref{non-bonded}.
\end{proof}

\subsection{Application to iterated $D$-products}

${}^{}$

Recall from 
section \ref{2.1} 
that for any group automorphism $F: G \to G$ and any integer $t \geq 0$ we use 
the following notation:
$$F^{(t)}(g) := g F(g) F^2(g) \ldots F^{t-1}(g)$$
and
$$F^{(s,t)}(g) := F^{s}(g) F^{s+1}(g) \ldots F^{t-1}(g)$$
for any exponents 
$s < t$.

\begin{prop}
\label{bounded-cancellation}
Let $D: \cal G \to \cal G$ be an efficient Dehn twist, and let $W_1, W_2 \in \Pi(\cal G)$ be reduced words based on 
non-trivial 
loops.

Assume that 
for some $U \in \Pi(\cal G)$ and  $t \geq 0$ the cancellation (with respect to $\cal G$-length) in the family of products
\begin{equation}
\label{iterated-product}
D^{(t)}_*(W_1) \cdot D_*^t(U) \cdot (D^{(t)}_*(W_2))^{-1}
\end{equation}
is unbounded.
Then $W_1$ and $W_2$ are $D$-conjugate to each other. More precisely, one has:
$$U^{-1} W_1 D_*(U) = W_2$$
\end{prop}

\begin{proof}
From the claimed statement we observe 
that one can assume without loss of generality that $W_1$ and $W_2$ are $D$-reduced, and that $U$ (and hence any $D_*^t(U)$) is reduced. 
Thus the unboundedness hypothesis of the cancellation in 
(\ref{iterated-product}) 
implies that there is a reduced product decomposition $U = U_1 U_2$ such that $D_*^t(U_1)$ cancels completely against the end of $D^{(t)}_*(W_1)$, and $D_*^t(U_2^{-1})$ against the end of $D^{(t)}_*(W_2)$.
Hence, through possibly replacing $W_1$ by $U_1 W_1 D_*(U_1^{-1})$ and $W_2$ by $U_2^{-1} W_2 D_*(U_2)$, we can furthermore assume the $U = 1$.
By doing these replacements iteratively, where at each step the replacement $U'_1 W_1 D_*({U'_1}^{-1})$ is taken such that $U'_1$ is an initial subword of $U_1$ of $\cal G$-length $|U'_1| \leq  |W_1|$ (and analogously for a terminal subword $U'_2$ of $U_2$) we see that these replacements will preserve the property that both, $W_1$ and $W_2$ are $D$-reduced.

By checking the $\cal G$-lengths $m = |W_1|$ and $n = |W_2|$, unbounded cancellation in the products 
$$P(t):= D^{(t)}_*(W_1) \cdot 
(D^{(t)}_*(W_2))^{-1}$$
implies that 
for some sufficiently large $t$ the suffix  
$D_*^{(t-n,t)}(W_1)$
of $D^{(t)}_*(W_1)$ 
cancels against $D_*^{(t-m,t)}(W_2)^{-1}$,
and subsequently $D_*^{(t-2n,t-n)}(W_1)$ against $D_*^{(t-2m,t-m)}(W_2)^{-1}$, and so on. However, since $D_*^{(t-n,t)}(W_1) = $ $D_*^n(D_*^{(t-2n,t-n)}(W_1))$ and $D_*^{(t-m,t)}(W_2) = D_*^m(D_*^{(t-2m,t-m)}(W_2))$, by Lemma \ref{D-image-cancellation} this implies $m = n$, or equivalently, $|W_1| = |W_2|$.

This implies that $D_*^{t-1}(W_1)$ cancels against $D_*^{t-1}(W_2)^{-1}$, and $D_*^{t-2}(W_1) u$ cancels against $D_*^{t-2}(W_2)^{-1}$, for $u := D_*^{t-1}(W_1) (D_*^{t-1}(W_2))^{-1}$ of $\cal G$-length $|u| = 0$. Since $D_*^{t-1}(W_1) = D(D_*^{t-2}(W_1))$ and $D_*^{t-1}(W_2) = D(D_*^{t-2}(W_2))$, 
it follows (see Remark \ref{cancellation-isomorphism}) that $D_*^{t-2}(W_1)$ cancels against $D_*^{t-2}(W_2)^{-1}$, so that
Lemma \ref{unique-cancellation} (applicable by Lemma \ref{non-bonded}) implies $u = 1$. Hence we obtain indeed $D_*^{t-1}(W_1) = D_*^{t-1}(W_2)$ and thus $W_1 = W_2$.
\end{proof}

\subsection{Limit lengths}

${}^{}$

For any basis $\cal A = \{a_1, \ldots, a_N\}$ of $\FN$ 
and any $w \in \FN$ we denote by $|w|_\cal A$ the length of the reduced word in $\cal A^{\pm 1} = \{a_1, a_1^{-1}, \ldots, a_N, a_N^{-1}\}$ representing $w$, and by $\|w\|_\cal A$ the length of 
any 
cyclically reduced word 
representing the conjugacy class $[w]$.

Below one considers, for any graph-of-groups $\cal G$ and any vertex $v$ of $\cal G$, an identification $\pi_1(\cal G, v) =\FN$, and we choose any basis $\cal A$ of $\FN$. For any edge $e$ of $\cal G$, any element $z \in G_e$ and any connected word $W \in \Pi(\cal G)$ starting at $v$ and terminating at $\tau(e)$, we will write for simplicity $\|z\|_\cal A$ instead of $ \|W f_e(z)  W^{-1}\|_\cal A$ or $\|W t_e^{-1} f_{\bar e}(z) t_e W^{-1} \|_\cal A$, which would be formally correct, but clearly 
the conjugacy class of $ W f_e(z)  W^{-1} = W t_e^{-1} f_{\bar e}(z) t_e W^{-1}$
depends only on $z$.

The following length estimate has been shown in 
\cite{Ye02}.
A slightly weaker but much related result is given by Proposition 6.19 of \cite{R}.

\begin{prop}
\label{quadratic-limit}
Let $D: \cal G \to \cal G$ be 
an efficient Dehn twist with 
twistors $z_e$ 
for any edge $e$ of $\cal G$, 
and let $v$ be a vertex of $\cal G$. For the identification $\pi_1(\cal G, v) \cong \FN$ we denote by $\cal D \in \Aut(\FN)$ the automorphism induced by $D$. Let $W = w_0 t_1 w_1 \ldots w_{r-1} t_r w_r$ be a 
$D$-reduced 
word in $\Pi(\cal G)$
based on a 
non-trivial 
closed connected path 
$\gamma = e_1 e_2 \ldots e_r$, and let $V \in \Pi(\cal G)$ 
and $g \in \FN$ 
be 
such that 
$V^{-1} W D(V)$ 
is an element in $\pi_1(\cal G, v)$ which represents 
$g$. 
Then one has, 
for any basis $\cal A$ of $\FN$:
$$\lim_{t \to \infty} \frac{|\cal D^{(t)}(g)|_\cal A}{t^2} = \frac{1}{2} \sum_{i = 1}^{r} \| 
z_{e_i} 
\|_\cal A .$$
\end{prop}

\begin{rem}
\label{Ye02-modifs}
(1)
In the present version of \cite{Ye02} the equality from Proposition \ref{quadratic-limit} is 
not stated precisely as given here. However, in the proof of Proposition 6.5 of \cite{Ye02} all arguments are given, except that in the last paragraph the estimation for the upper bound has to be taken slightly more sharply.

\smallskip
\noindent
(2)
Also, the cancellation arguments from the proof of Proposition 6.5 in \cite{Ye02} apply as well to a basis $\cal A$ of any larger free group $\FN$ which contains 
$\pi_1 (\cal G, v)$ 
as free factor.
\end{rem}

\begin{rem}
\label{length-zero}
If in Proposition \ref{quadratic-limit} the path $\gamma$ is trivial, then $W = w_0$ has length $|W| = 0$ and belongs to one of the vertex groups of $\cal G$, which are element-wise fixed by the induced automorphism $D_*$. Hence $D^{(t)}_*(W) = W^t$, so that $|\cal D^{(t)}(g)|_\cal A$ grows at most linearly in $t$.
On the other hand, for such $\gamma$ the sum $\sum_{i = 1}^{r} \| 
z_{e_i} 
\|_\cal A$
is taken over the empty set and hence equal to 0, so that the conclusion of Proposition \ref{quadratic-limit} holds in this case as well.
\end{rem}

We now use the crucial Proposition \ref{quadratic-limit} and combine it with the cancellation results from the previous subsection, to obtain:

\begin{prop}
\label{limit-bounded}
Let $D: \cal G \to \cal G$ be an efficient Dehn twist with twistors $(z_e)_{e \in E(\cal G)}$, 
and let $v$ be a vertex of $\cal G$. 
For the identification $\pi_1(\cal G, v) \cong \FN$ let $\cal D \in \Aut(\FN)$ the automorphism induced by $D$. Let $\cal A$ be any basis of a free group $F_M$ which contains $\FN$ as free factor.


Let $U, V$ and $W$ be reduced words in $\pi_1(\cal G, v) \subset \Pi(\cal G)$, and
let $e_1 e_2 \ldots e_r$ and $e'_1 e'_2 \ldots e'_{r'}$ be the 
(possibly trivial)
loops on which 
$D$-reduced 
words are based that are 
$D$-conjugate 
to $V$ and $W$ respectively. 
If these two loops are non-trivial, we 
assume furthermore that 
$U^{-1}VD_*(U) \neq W$.

Then one has, for $g_1, g_2, h \in \FN$ representing $V, W$ and $U$ respectively:
$$\lim_{t \to \infty} \frac{|\cal D^{(t)}(g_1) \cal D^{t}(h) (\cal D^{(t)}(g_2)^{-1}|_\cal A}{t^2} = \frac{1}{2} (\sum_{i = 1}^{r} \| 
z_{e_i} 
\|_\cal A + \sum_{i = 1}^{r'} \| 
z_{e'_i} 
\|_\cal A)$$
\end{prop}

\begin{proof}
From the assumed inequality $U^{-1}VD_*(U) \neq W$ and from Proposition \ref{bounded-cancellation} we deduce that the cancellation in the products $D^{(t)}_*(V) \cdot D_*^t(U) \cdot (D^{(t)}_*(W))^{-1}$, and hence in the products 
$\cal D^{(t)}(g_1) \cdot \cal D^{t}(h) \cdot (\cal D^{(t)}(g_2)^{-1}$, 
is bounded independently of $t$. Furthermore, $\cal D^t(h)$ grows at most linearity in $t$. Hence, both, the possible cancellation in the products as well as the factor $\cal D^{t}(h) $ can be neglected when taking the limit quotient modulo $t^2$. Thus the desired equality follows directly from Proposition \ref{quadratic-limit}
and Remark \ref{length-zero}.
\end{proof}

\begin{rem}
\label{D-length}
${}^{}$
The set-up given in the first paragraph of Proposition \ref{limit-bounded} will be encountered frequently in the next section, as local graph-of-groups isomorphism. We hence want to formalize slightly some of the above:

\smallskip
\noindent
(1)
Let $g \in \FN \cong \pi_1(\cal G, v)$, and let 
$V = w_0 t_1 w_1 \ldots w_{r-1} t_r w_r$ be a connected, closed and  
$D$-reduced 
word in $\Pi(\cal G)$
based on a 
(possibly trivial) 
path 
$\gamma = e_1 e_2 \ldots e_r$, and let $U \in \Pi(\cal G)$ 
be such that 
$W := U^{-1} V D(U)$ 
is an element in $\pi_1(\cal G, v)$ which represents 
$g$. 
We observe that the limit quotient for $g$ considered in Proposition \ref{quadratic-limit} does not depend on $W$, but only on $V$, or more precisely, on the $D$-conjugacy class of $W$.

\smallskip
\noindent
(2)
We thus define for any closed connected word $W \in \Pi(\cal G)$
the 
{\em $D$-length}
of the $D$-conjugacy class $[W]_D$ by
$$\|[W]_{D}\|_\cal A := \frac{1}{2} \sum_{i = 1}^{r} \| 
z_{e_i} 
\|_\cal A \, ,$$
where $e_1 e_2 \ldots e_r$ is the path in $\Gamma(\cal G)$ underlying any $D$-reduced word 
$V$ which is $D$-conjugate to $W$.

\smallskip
\noindent
(3)
Since the $D$-conjugacy class of $W$ consists of all inverses of the $D^{-1}$-conjugacy class of $W^{-1}$ (see Definition \ref{twisted-conjugacy} and the subsequent paragraph), and since the twistors for $D^{-1}$ are given by the inverses $z_e^{-1}$ of the twistors $z_e$ of $D$, which satisfy furthermore $\|z_e^{-1}\| = \|z_e\|$, one has:
$$\|[W^{-1}]_{D^{-1}}\|_\cal A = \|[W]_D\|_\cal A$$
\end{rem}


\section{Parabolic dynamics on Outer space}


Let 
$H: \cal G \to \cal G$ be an efficient 2-level Dehn twist as in 
section \ref{sec-3}, 
and let $W = w_0 t_1 w_1 \ldots w_{q-1} t_q w_q$ be a reduced word in $\Pi(\cal G)$. Then 
from Remark \ref{long-formula} (2) we know that 
for any $t \in \Z$ one has 
$$
H_*^t(W) = w_0(t) t_1 w_1(t) \ldots w_{q-1}(t) t_q w_q(t)\, ,$$ 
with 
\begin{equation}
\label{syllable-growth}
w_k(t) = H_k^{(t)}(\delta_k^{-1}) H_k^t(w_k) H_k^{(t)}
(\bar\delta_{k+1}^{-1})^{-1} 
\end{equation}
for any $k = 0, \ldots, q$.
Here we formally define $\delta_0 = \bar \delta_{q+1} = 1$, and otherwise adopt the convention that $\delta_k$ denotes the correction term of the edge $E_k$ corresponding to the stable letter $t_k$, while $\bar \delta_k$ denotes the correction term of $\bar E_k$. Furthermore, $H_k$ denotes the vertex group isomorphism $H_{\tau(E_k)}$.
We also define 
\begin{equation}
\label{q,0}
w_{q,0}(t) := w_q(t) w_0(t) = 
H_q^{(t)}(\delta_q^{-1}) H_q^t(w_q w_0) H_q^{(t)} (\bar\delta_{1}^{-1})^{-1}
\end{equation}


\begin{prop}
\label{conjugacy-limit-growth}
Let $H: \cal G \to \cal G$ be an efficient 2-level Dehn twist which induces on $\pi_1 \cal G \cong \FN$ the automorphism $\phi \in \Out(\FN)$. For any $g \in \FN$ let $W = w_0 t_1 w_1 \ldots w_{q-1} t_q w_q$ be a cyclically reduced 
word in $\Pi(\cal G)$ which represents the conjugacy class $[g] \subset \FN$. Then, for any basis $\cal A$ of $\FN$, we have
$$\lim_{t \to \infty} \frac{\|\phi^{t}([g])\|_\cal A}{t^2} = 
\sum_{i = 1}^{q} \|[\delta_i]_{D_i}\|_\cal A
\, ,$$
where for any stable letter $t_i$ in $W$ we denote by $\delta_i$ the 
correction term $\delta(E_i)$ or $\delta(\bar E_i)$ 
for
the edge $E_i$ associated to $t_i$,
according to which of the two, $E_i$ or $\bar E_i$,
is not locally zero. 
By $D_i$ we denote 
here 
the efficient Dehn twist which represents $H_{\tau(E_i)}$ or $H_{\tau(\bar E_i)}$ respectively.
\end{prop}

\begin{proof}
We consider the cyclically reduced word $W$ which represents the conjugacy class $[g]$ and its images $H^t(W)$ representing $\phi^t([g])$, for increasing exponents $t \geq 0$ (or, similarly, for decreasing exponents $t \leq 0$). Since the vertex groups of $\cal G$ are free factors of $\pi_1 \cal G$ and of $\Pi(\cal G)$, any cancellation between adjacent elements of distinct vertex groups, or between distinct conjugates of the same vertex group, are a priori bounded, for any chosen basis $\cal A$ of $\FN$. Hence we can consider separately the growth of each ``syllable'' $t_k w_k t_{k+1}$ in $W$ under iteration of $H$, which has been described above in (\ref{syllable-growth}) by the elements $w_k(t)$ (while the stable letters stay constant and can hence be ignored when passing to the limit quotient by $t^2$).
Since $W$ is furthermore assumed to be cyclically reduced, we can also similarly consider the ``cyclic syllable�� $t_q w_q w_0 t_0$ and its $H$-itarations given through $w_{q,0}(t)$ in (\ref{q,0}), to obtain:
$$\lim_{t \to \infty} \frac{\|\phi^{t}([g])\|_\cal A}{t^2} = 
{\big (}\sum_{k = 1}^{q-1} 
\lim_{t \to \infty} \frac{|w_k(t)|_\cal A}{t^2} {\big )}
+ \lim_{t \to \infty} \frac{|w_{q,0}(t)|_\cal A}{t^2}
$$
For each of the $w_k(t)$ we can apply 
Proposition \ref{limit-bounded}, which gives 
as limit quotient 
$\|[\delta(E_k)^{-1}]_{D^{-1}_{\tau(E_k)}}\|_\cal A + \|[\delta(\bar E_{k+1})^{-1}]_{D^{-1}_{\tau(E_k)}}\|_\cal A$, which by Remark \ref{D-length} (3) is equal to
$\|[\delta(E_k)]_{D_{\tau(E_k)}}\|_\cal A + \|[\delta(\bar E_{k+1})]_{D_{\tau(E_k)}}\|_\cal A$.
Similarly, for $w_{q,0}(t)$ we obtain the limit quotient
$\|[\delta(E_q)]_{D_{\tau(E_q)}}\|_\cal A + \|[\delta(\bar E_{1})]_{D_{\tau(E_q)}}\|_\cal A$.
We note here that the hypothesis $U^{-1} V D_*(U) \neq W$ from Proposition \ref{limit-bounded} is satisfied since for $E_k \neq \bar E_{k+1}$ the correction terms $\delta(E_k)$ and $\delta(\bar E_{k+1})$ are not $D_{\tau(E_k)}^{-1}$-conjugate, by the definition of an efficient 2-level Dehn twist (see Definition \ref{2-level-efficient}). For $E_k = \bar E_{k+1}$ we use the hypothesis that $W$ is reduced. The analogous arguments apply to $E_q$ and $\bar E_1$, since $W$ is furthermore cyclically reduced.
We thus have:
\begin{align*}
\lim_{t \to \infty} \frac{\|\phi^{t}([g])\|_\cal A}{t^2} 
& ={\big (}\sum_{k = 1}^{q-1} 
\|[\delta(E_k)]_{D_{\tau(E_k)}}\|_\cal A + \|[\delta(\bar E_{k+1})]_{D_{\tau(E_k)}}\|_\cal A {\big )}\\
&\qquad \qquad \qquad  + \|[\delta(E_q)]_{D_{\tau(E_q)}}\|_\cal A + \|[\delta(\bar E_{1})]_{D_{\tau(E_q)}}\|_\cal A\\
&= \sum_{i = 1}^{q} 
\|[\delta(E_i)]_{D_{\tau(E_i)}}\|_\cal A + \|[\delta(\bar E_{i})]_{D_{\tau(\bar E_i)}}\|_\cal A
\end{align*}

Since from the definition of an efficient 2-level Dehn twist we know that for each edge $E$ of $\cal G$ precisely one of 
$E$ or $\bar E$
is locally zero
(and hence 
$\delta(E)$ or $\delta(\bar E)$
has zero $D_V$-length, for $V = \tau(E)$ or $V = \tau(\bar E)$ respectively), 
the last sum 
gives 
precisely 
the desired result.
\end{proof}

\begin{rem}
\label{other-lengths}
For any length function $\| \cdot \|$ on the conjugacy classes of $\FN$ which is induced by a length function that is quasi-isometric to the one given by any basis $\cal A$ of $\FN$, the equation
$$\lim_{t \to \infty} \frac{\|\phi^{t}([g])\|}{t^2} = 
\sum_{i = 1}^{q} \|[\delta_i]_{D_i}\|
$$
analogous to the result in Proposition \ref{conjugacy-limit-growth} stays valid, 
since the linear quasi-isometry constants disappear in the limit when considering the quotient by $t^2$,
and the previously used cancellation arguments apply to $\| \cdot \|$ as well.

\end{rem}

Remark \ref{other-lengths} applies in particular to translation length functions $\| \cdot \|_{\tilde \Gamma}$ on a metric tree $\tilde \Gamma$ given as universal covering of a metric graph $\Gamma$ equipped with a marking isomorphism $\theta: \pi_1 \Gamma \to \FN$. Such length functions define, after passing to the projective class $[\Gamma]$, a point in 
Outer space $\CVN$. 
For background on Outer space $\CVN$ 
and its ``Thurston compactification'' $\CVNbar = \CVN \cup \partial \CVN$
we refer the reader to \cite{Vog}.

On the other hand, any graph-of-groups $\cal G$ with marking isomorphism $\pi_1 \cal G \cong \FN$ defines a simplex $\Delta_\cal G$ in the boundary $\partial \CVN$ of $\CVN$ (or in $\CVN$, if all vertex groups of $\cal G$ are 
trivial), where a point in $\Delta_\cal G$ is 
given by defining 
an edge length $L(E) \geq 0$ 
for any edge $E$ 
of $\cal G$,
with $L(E) = L(\bar E)$.
Thus, for $E^+(\cal G)$ as 
given 
in Definition \ref{2-level-efficient} (2), 
the Bass-Serre tree $T_{(\cal G, (L(E))_{E \in E^+(\cal G)})}$ 
associated to the point $[\cal G, (L(e))_{E \in E^+(\cal G)}] \in \Delta_\cal G$ becomes a metric simplicial tree, equipped canonically with an action of $\FN$ by isometries. 
%
We obtain:

\begin{thm}
\label{parabolic-dynamics}
Let 
$[\Gamma]$
be any point in Outer space $\CVN$, given by a 
marked metric graph $\Gamma$.
Then for any automorphism $\phi \in \Out(\FN)$, 
represented by an efficient 2-level Dehn twist $D: \cal G \to \cal G$, 
the $\phi$-orbit of 
$[\Gamma]$ 
is parabolic, with 
limit point contained 
in the 
interior of the 
simplex $\Delta_\cal G \subset \partial\CVN$.
More precisely, one has:
$$\lim_{t \to \pm \infty} \phi^t(
[\Gamma]
) = [\cal G, (\|[\delta(E)]_{D_{\tau(E)}}\|_{
\tilde \Gamma})_{E \in E^+(\cal G)}]$$
\end{thm}

\begin{proof}
This is a direct consequence of Proposition \ref{conjugacy-limit-growth}, where for the ``interior point'' 
claim 
we observe that for any point $[\Gamma]$ in $\CVN$ any non-trivial conjugacy class has non-zero $\tilde \Gamma$-length,
so that for any edge $E$ which is not locally zero one has $\|[\delta(E)]_{D_{\tau(E)}}\|_{\tilde \Gamma} > 0$.
\end{proof}

\begin{rem}
\label{interior-points}
The ``interior points'' statement in Theorem \ref{parabolic-dynamics} is useful, since 
if two 
marked graph-of-groups $\cal G$ and $\cal G'$ give rise to distinct
simplexes $\Delta_\cal G \neq \Delta_{\cal G'}$ in $\CVNbar$,
then $\Delta_\cal G$ and 
$\Delta_{\cal G'}$ 
can not
intersect in points that are interior in both $\Delta_\cal G$ and 
$\Delta_{\cal G'}$.

This follows from the fact that for any interior point $[\cal G, (L(e))_{E \in E^+(\cal G)}] \in \Delta_\cal G$ all edges $e$ have length $L(e) > 0$, so that the associated tree $T_{(\cal G, (L(E))_{E \in E^+(\cal G)})}$, after forgetting the metric, is $\FN$-equivariantly homeomorphic to the Bass-Serre tree $T_\cal G$ associated to $\cal G$.

\end{rem}


\section{Normal from for 
quadratically growing automorphisms}

From the geometric result in the previous section we can now derive that efficient 2-level  
Dehn twists constitute indeed a normal from for the induced outer automorphisms:

\begin{thm}
\label{normal-form}
Two efficient 2-level Dehn twists $H:\cal G \to \cal G$ and $H':\cal G' \to \cal G'$ represent outer automorphisms $\hat H$ and $\hat H'$ of a free group $\FN$ which are conjugate in $\Out(\FN)$ if and only if there exists a graph-of-groups isomorphism $F: \cal G \to \cal G'$ which satisfies:
$$\hat H = \hat F^{-1} \hat H' \hat F$$
\end{thm}

\begin{proof}
The ``if'' direction is obvious. To show the ``only if'' direction we note that any conjugating automorphism $\psi \in \Out(\FN)$, with $\hat H = \psi^{-1} \hat H' \psi$, must map any $\hat H$-orbit of the 
$\Out(\FN)$-action on $\CVN$ to an $\hat H'$-orbit, 
and hence the limit point of the former to the limit point of the latter.
It follows from Remark \ref{interior-points} that $\psi$ maps
the limit simplex $\Delta_{\cal G}$ from Theorem \ref{parabolic-dynamics} to the analogous simplex $\Delta_{\cal G'}$. 

In particular, 
the automorphism $\psi$ maps the center point of $\Delta_{\cal G}$, defined by setting all edge lengths equal to 1, to the analogously defined center point of $\Delta_{\cal G'}$. Hence $\psi$ conjugates the $\FN$-action on the non-metric Bass-Serre tree $T_{\cal G}$ to that on the analogous tree $T_{\cal G'}$, thus inducing a graph-of-groups isomorphism $F: \cal G \to \cal G'$ which satisfies $\hat F = \psi$. This last conclusion is a standard fact for graph-of-groups, see for instance Lemma 4.5 of \cite{CL99}
or Proposition 4.4 of \cite{Bass93}.
%
\end{proof}

\begin{rem}
\label{left-right}
The natural group action of $\Out(\FN)$ on Outer space $\CVN$ is a right action, but of course it can be canonically transformed into a left action by setting $\phi \cdot [\Gamma] := [\Gamma] \cdot \phi^{-1}$. The watchful reader will notice that in either case, in the above setting, the conjugating automorphism $\psi$ maps $\hat D$-orbits to a $\hat D'$-orbits.
\end{rem}


In \cite{Ye03} {\em iterated Dehn twists of level $k \geq 1$} have been introduced as certain iteratively defined graph-of-groups automorphisms $H: \cal G \to \cal G$. For $k = 1$ one obtains ordinary Dehn twists, and for $k=2$ this notion agrees with what 
is called here
``2-level Dehn twists''.
It has been shown in Proposition 1.1 of \cite{Ye03} that every polynomially growing automorphism $\phi \in \Out(\FN)$ has a positive power which is represented by an iterated Dehn twist of some level $k \geq 1$.

Furthermore, 
it can be shown 
that any iterated Dehn twist $H: \cal G \to \cal G$ of level $k \geq 2$ represents an automorphism $\phi$ which either can be represented by an  iterated Dehn twist of level $k - 1$, or else there is a conjugacy class in $\FN$ which has polynomial growth of degree precisely equal to $k$, under iteration of $\phi$.  This has been shown in \cite{Ye02} for $k=2$; a proof for $k \geq 3$ is obtained by analogous arguments, except that it becomes easier since the graph-of-groups in question for $k \geq 3$ have trivial edge groups
(compare also \cite{R}, Proposition 4.21).
Hence every quadratically growing automorphism $\phi \in \Out(\FN)$ 
has a positive power which is represented by an efficient 2-level Dehn twist $H: \cal G \to \cal G$.

The following extension of Theorem \ref{normal-form} gives hence a normal form for any quadratically growing automorphism of $\FN$.

\begin{cor}
\label{roots-normal-form}
(1)
Every automorphism $\phi \in \Out(\FN)$ with exponent 
$m \geq 1$ 
such that $\phi^m$ is represented by an efficient 2-level Dehn twist $H: \cal G \to \cal G$ can be represented by a graph-of-groups automorphism $R: \cal G \to \cal G$.

\smallskip
\noindent
(2)
Two graph-of-groups automorphism $R: \cal G \to \cal G$ and $R': \cal G' \to \cal R'$ as in part (1) represent outer automorphisms $\hat R$ and $\hat R'$ of a free group $\FN$ which are conjugate in $\Out(\FN)$ if and only if there exists a graph-of-groups isomorphism $F: \cal G \to \cal G'$ which satisfies:
$$\hat R = \hat F^{-1} \hat R' \hat F$$
\end{cor}

\begin{proof}
(1) From the hypothesis $\phi^m = \hat H$ we see that $\phi$ permutes the $\hat H$-orbits in $\CVN$, and hence also their limit points. But since by Theorem \ref{parabolic-dynamics} all of the latter are contained in the interior of the simplex $\Delta_\cal G$, it follows from Remark \ref{interior-points} that $\phi$ must map $\Delta_\cal G$ to itself, and hence fix its center point. Hence the claim follows by the exactly the same argument as given in the second paragraph of the proof of Theorem \ref{normal-form}.

\smallskip
\noindent
(2) As in the proof of Theorem \ref{normal-form} and of part (1) above, any automorphism $\psi \in \Out(\FN)$ with $\hat R = \psi^{-1} \hat R' \psi$ must map $H$-orbits in $\CVN$ to $H'$-orbits, where 
$H$ and $H'$ are positive powers of $R$ and $R'$ respectively which 
are efficient 2-level Dehn twists. As in the proof of Theorem \ref{normal-form} we deduce the existence of the desired graph-of-groups isomorphism $F$, and hence the ``only if'' part of the claim. The ``if'' part is again obvious.
\end{proof}

\begin{rem}
\label{conjugacy-split}
${}^{}$
It is well known (see for instance section 2 of \cite{Ye01}) that the relationship between 
graphs-of-groups and their isomorphisms on one hand, and simplicial trees with equivariant homeomorphisms on the other, though conceptually very appealing, is on a technical level not quite as smooth as one would wish.  The ``fault'' lies entirely on the graph-of-groups side, which is technically loaded with data that are non-uniquely determined by the corresponding Bass-Serre trees.

For example, if $T$ is a simplicial tree with $G$-action, and $\cal G$ and $\cal G'$ two graph of groups with $\pi_1 \cal G = \pi_1 \cal G' = G$ and $G$-equivariant identifications $T = T_\cal G = T_{\cal G'}$, then there is indeed a graph-of-groups isomorphism $F: \cal G \to \cal G'$ which induces on the fundamental group the above identification.

Furthermore, if $\tilde H: T \to T$ a homeomorphism which commutes with an automorphism $\Phi: G \to G$ (in the sense that for any $g \in G$ one has $\tilde H g = \Phi(g) \tilde H: T \to T$), then $\tilde H$ descends to graph-of-groups automorphisms $H: \cal G \to \cal G$ and $H': \cal G' \to \cal G'$ which satisfy: 
$$\hat H = \hat F^{-1} \hat H' \hat F$$

However, whether $F$ (or $F$, $H$ and $H'$) can be chosen so that one obtains actually
$$ H = F^{-1} H' F \, ,$$
remains a question to which even under natural additional assumptions (like minimality of $T$) an answer seems in general not to be known.
\end{rem}




















\end{document}